\documentclass[10pt]{amsart}

\usepackage{amssymb}
\usepackage{stmaryrd}
\usepackage{mathrsfs}
\usepackage{array,subfigure}
\usepackage{graphicx}

\newtheorem{thm}{Theorem}[section]
\newtheorem{prop}[thm]{Proposition}
\newtheorem{cor}[thm]{Corollary}
\newtheorem{lem}[thm]{Lemma}
\newtheorem{defn}[thm]{Definition}

\newcounter{xpl}[section]

\newenvironment{xpl}{\refstepcounter{xpl}  \medskip \noindent {\bf  Example \thesection.\arabic{xpl}}}{\hfill$\diamondsuit$\mbox{}\bigskip}

\newenvironment{remark}{\refstepcounter{thm} \medskip \noindent {\bf  Remark \arabic{section}.\arabic{thm}}}{\hfill\mbox{}\bigskip}

\newcounter{thmlist}

\newenvironment{thmlist}{\begin{list}{(\roman{thmlist})}{\usecounter{thmlist}\setlength{\leftmargin}{25pt}
\setlength{\itemindent}{0pt}\setlength{\labelwidth}{20pt}\setlength{\labelsep}{5pt}\setlength{\itemsep}{0in}}}{\end{list}}

\newcommand{\C}{\mathbb{C}}

\newcommand{\R}{\mathbb{R}}
\newcommand{\N}{\mathbb{N}}
\newcommand{\Z}{\mathbb{Z}}
\newcommand{\Q}{\mathbb{Q}}
\newcommand{\rps}{\mathbb{R}P}
\newcommand{\cps}{\mathbb{C}P}

\newcommand{\ol}[1]{\bar{#1}}
\newcommand{\sm}[1]{\scriptscriptstyle{#1}}
\newcommand{\contr}{\,\lrcorner\,}
\newcommand{\codim}{\operatorname{codim}}
\newcommand{\cone}{\operatorname{Cone}}
\newcommand{\diff}{\operatorname{Diff}}
\newcommand{\dist}{\operatorname{dist}}

\newcommand{\Aut}{\operatorname{Aut}}
\newcommand{\Hom}{\operatorname{Hom}}
\newcommand{\im}{\operatorname{Im}}

\newcommand{\ric}{\operatorname{Ricci}}
\newcommand{\Ric}{\operatorname{Ric}}

\newcommand{\ind}{\operatorname{Ind}}

\newcommand{\inj}{\operatorname{inj}}

\newcommand{\rank}{\operatorname{rank}}

\newcommand{\Sing}{\operatorname{Sing}}
\newcommand{\inter}{\operatorname{Int}}

\title{A Construction of Complete Ricci-flat K\"{a}hler Manifolds}
\author{Craig van Coevering}
\address{Department of Mathematics, Massachusetts Institute of Technology, 77 Massachusetts Avenue, Cambridge, MA 02139-4307}
\email{craig@math.mit.edu}
\date{December 5, 2007}

\begin{document}

\begin{abstract}
We consider an extension of the results of S. Bando, R. Kobyashi, G. Tian, and S. T. Yau on the existence of Ricci-flat
K\"{a}hler metrics on quasi-projective varieties $Y=X\setminus D$ with $\alpha[D]=c_1(X), \alpha >1$.
The requirement that $D$ admit a K\"{a}hler-Einstein metric is generalized to the condition that the
link $S\subset N_D$ in the normal bundle of $D$ admits a Sasaki-Einstein structure in the Sasaki-cone
of the usual Sasaki structure provided the embedding $D\subset X$ satisfies an additional holomorphic condition.
If $D$ is a toric variety, then $S$ always admits a Sasaki-Einstein metric.  As an application we prove that
every small smooth deformation of a toric Gorenstein singularity admits a complete Ricci-flat K\"{a}hler metric asymptotic
to a Calabi ansatz metric.  Some examples are given which were not previously known.
\end{abstract}

\maketitle

\section{introduction}

The purpose of this article is to extend the theorem due to G. Tian and S.-T. Yau,
and independently by S. Bando and R. Kobyashi, which gives the existence of a complete Ricci-flat K\"{a}hler metric on a quasi-projective manifold $Y=X\setminus D$ under the assumption that $D$ admits a K\"{a}hler-Einstein metric.
Here $X$ is a projective variety, and $D$ is irreducible and supports the anti-canonical divisor.
Further, $X$ and $D$ are required to be K\"{a}hler orbifolds.  An assumption that $X$ and $D$ are manifolds would
restrict the examples of smooth $Y=X\setminus D$ with a Ricci-flat metric that this technique produces.
For the definition of a K\"{a}hler orbifold and the notions of divisors and line bundles on orbifolds see
~\cite{Ba}.  We first review the original result.

Let $X$ be a compact K\"{a}hler orbifold, with $\dim_\C X=n$, and with $\dim_\C(\Sing X)\leq n-2$.
Suppose there is a divisor $D\subset X$ such that $\alpha[D]=-K_X$, with $\alpha >1$.
We will need the following.
\begin{defn}\label{defn:good-div}
Let $D$ be a divisor on a compact K\"{a}hler orbifold.  Then
\begin{thmlist}
    \item $D$ is \emph{admissible} if $\Sing X\subset D$ and for any local uniformizing chart
    $\pi:\tilde{U}\rightarrow U$ at $x\in D$, $\pi^{-1}(D)$ is smooth in $\tilde{U}$.
    \item $D$ is \emph{almost ample} if there is an integer $k\gg 0$ such that the divisor $kD$ defines \label{item:al-am}
    a morphism $\iota_{\sm kD}:X\rightarrow\cps^N$ which is biholomorphic in a neighborhood of $D$ onto its image.
\end{thmlist}
We will call $D$ \emph{good} if it is admissible and almost ample.
\end{defn}

In~\cite{TY2} the following is proved.  See also~\cite{BK1,BK2} and~\cite{TY1} for similar results.
\begin{thm}\label{thm:C-Y}
Let $X$ be a K\"{a}hler orbifold, and let $D$ be a good divisor with $\alpha[D]=-K_X, \alpha >1.$   Suppose that $D$
admits a K\"{a}hler-Einstein metric, then there exists a complete Ricci-flat K\"{a}hler metric $g$ on $Y=X\setminus D$
in every K\"{a}hler class in $H^2_c(Y,\R)$.

Furthermore, if $\rho$ denotes the distance function on $Y$ from a fixed point and $R_g$ denote the curvature
tensor of $g$,  then $\|\nabla ^k R_g\|_g =O(\rho^{-2-k})$.
\end{thm}

We denote by $H^2_c(Y,\R)$ cohomology with compact support.
The metrics in the theorem have Euclidean volume growth.  It follows from the results of~\cite{BKN}
that if $\|R_g\|_g =O(\rho^{-j})$ for $j>2$, then $Y$ is asymptotically locally Euclidean (ALE).

Recall the idea behind theorem~\ref{thm:C-Y}.  Choose a hermitian metric on $[D]$ with curvature $\omega_0$, whose
restriction $\omega_D =\omega_0 |_D$ defines a K\"{a}hler-Einstein metric on $D$ with
$\ric(\omega_D)=(\alpha -1)\omega_D$.  Let $\sigma$ be a section of $[D]$ vanishing on $D$, and
let $t=\log\|\sigma\|^{-2}$.  Then define the K\"{a}hler metric on $Y=X\setminus D$
\begin{equation*}
\begin{split}
 \omega & = \frac{n}{\alpha-1}i\partial\ol{\partial}\|\sigma\|^{-\frac{2(\alpha-1)}{n}}\\
    & = \|\sigma\|^{\frac{-2(\alpha-1)}{n}}\omega_0 +\frac{(\alpha-1)}{n}\|\sigma\|^{-\frac{2(\alpha-1)}{n}}i\partial t\wedge\ol{\partial}t.\\
\end{split}
\end{equation*}
Then $\omega^n$ has a pole of order $2\alpha$ along $D$.  There exists an holomorphic n-form $\Omega$ on $X$
with a pole of order $\alpha$ along $D$.  The K\"{a}hler-Einstein condition implies that the function
$f=\log(\frac{\Omega\wedge\ol{\Omega}}{\omega^n})$ extends to a smooth function on $X$ constant on $D$, and can be made to
vanish reasonably rapidly along $D$.  Then the existence of the Ricci-flat metric on $Y=X\setminus D$ is
proved by solving a Monge-Amp\`{e}re equation similar to the compact case.

Of course, in general there is no guarantee that $D$ admits a K\"{a}hler-Einstein metric, as there are well known
obstructions to the existence of positive scalar curvature K\"{a}hler-Einstein metrics.
See~\cite{Mat,Fut1,Fut2} for obstructions involving the automorphism group, and~\cite{Ti} for further obstructions.

This article is concerned with extending theorem~\ref{thm:C-Y} to examples where $D$ does not admit a
K\"{a}hler-Einstein metric.  The approach is to consider the link in the conormal bundle $S\subset\mathbf{L}:=N^*_D =[-D]|_D$.
There is a standard Sasaki structure on $S$ that is Sasaki-Einstein if and only if $D$ is K\"{a}hler-Einstein.
But if $D$ does not admit a K\"{a}hler-Einstein metric, then $S$ may still admit a Sasaki-Einstein structure.  This structure
can arise by varying the Reeb vector field.  This is made precise in the notion of the Sasaki cone, in analogy with the K\"{a}hler
cone of a K\"{a}hler manifold.

In the theorem we will need to assume that
\begin{equation}\label{cond}
H^1(D,\Theta_X\otimes\mathcal{O}(-kD)|_D) =0,\text{ for all } k\geq 2.
\end{equation}
Let $\mathbf{L}=N^*_D =[-D]|_D$ be the conormal bundle of $D\subset X$.  In the following theorem we can weaken
condition~(\ref{item:al-am}) in Definition~\ref{defn:good-div} to just the normal bundle $N_D =[D]|_D$ being positive.

\begin{thm}\label{thm:main}
Suppose $X$ is a K\"{a}hler orbifold and $D\subset X$ is a good divisor with $\alpha[D]=-K_X$, $\alpha>1$.
Suppose (\ref{cond}) is satisfied, and suppose the link $S\subset\mathbf{L}$ admits a Sasaki-Einstein
structure in the Sasaki-cone of the standard Sasaki structure on $S\subset\mathbf{L}$.
Then $Y=X\setminus D$
admits a complete Ricci-flat K\"{a}hler metric $g$, in every K\"{a}hler class in $H^2_c(Y,\R)$, which is asymptotic to a
Calabi ansatz metric in the following sense.
For fixed $k\in\N$ and $\delta >0$ there is a neighborhood $U\subset Y$ of infinity and a diffeomorphism
$\phi:U\rightarrow V\subset\mathbf{L}$ to a neighborhood of infinity.  There is a Ricci-flat Calabi ansatz metric
$g_{\sm 0}$ on $V$ so that if we set $\ol{g}=\phi^*g_{\sm 0}$, then on $U$ we have
\begin{equation}\label{eq:asymp}
\nabla^j \left(g -\ol{g}\right) =O\left( \rho^{-2n+\delta-j}\right)\quad\text{for } j\leq k,
\end{equation}
where $\nabla$ is the covariant derivative of $\ol{g}$.
\end{thm}
\begin{remark}
One can eliminate the ``$\delta$'' in (\ref{eq:asymp}) by applying the proof of a non-compact version
of the Calabi Conjecture in~\cite[Theorem 3.10]{vC4} where analysis of the Laplacian is used to get the sharp convergence.

One can probably eliminate the restriction to K\"{a}hler classes in $H^2_c(Y,\R)$.  R. Goto was able to do this in the
special case of a crepant resolution of a Sasaki cone.  One motivation for doing this is that many of the examples $Y$,
such as a smoothing of a singularity, will be Stein.  And in this case $H^2_c(Y,\R)=0$, since $H^k(Y,\R)=0$ for $k>n$.
So Theorem~\ref{thm:main} just produces a Ricci-flat metric in the trivial K\"{a}hler class.
\end{remark}

Note that if $Y$ is Stein, e.g. $Y$ is an affine variety, then there is a cohomologically trivial K\"{a}hler metric.  So there is no difficulty finding a K\"{a}hler class in $H^2_c(Y,\R)$.

Of course it is desirable to remove the condition (\ref{cond}).  But the author does not know how to construct
the approximating metric in the proof without it.  This work was inspired by interesting recent results on
irregular Sasaki-Einstein manifolds such as the solution of the problem of the existence of Sasaki-Einstein structures
on toric Sasaki manifolds by A. Futaki, H. Ono, and G. Wang~\cite{FOW}.

For a source of examples we consider deformations of isolated toric Gorenstein singularities.  It is known~\cite{FOW} that such
a singularity itself has a nice Ricci-flat K\"{a}hler metric which is a metric cone $C(S)$ over a Sasaki-Einstein manifold $S$.
The versal deformation space of an isolated toric Gorenstein singularity was constructed by K. Altmann~\cite{Alt}.
These singularities are rigid in dimensions $n\geq 4$, but there are many interesting examples in dimension 3.

\begin{cor}\label{cor:main}
Let $Y$ be a small deformation of an isolated toric Gorenstein singularity.  If $Y$ is smooth, then it admits a complete Ricci-flat
K\"{a}hler metric as in Theorem~\ref{thm:main} which is invariant under a $T^{n-1}$-action.
\end{cor}

Note the deformed space $Y$ is no longer toric but admits a $(\C^*)^{n-1}$-action.  Corollary~\ref{cor:main} parallels that of
the author in~\cite{vC2} and~\cite{vC3} where it is proved that certain crepant resolutions of isolated toric Gorenstein singularities
admit Ricci-flat K\"{a}hler metrics which are asymptotic to the cone metric on $C(S)$.
Because of the $T^{n-1}$-action, the moment map construction of M. Gross~\cite{Gro} produces special Lagrangian fibrations
on the examples produced in Corollary~\ref{cor:main}.  These fibrations were studied in~\cite{Gro}; the contribution
here is in proving the existence of the Calabi-Yau metric.

One motivation for studying this problem is the conjecture, due to S.-T. Yau, that if
$Y$ is a complete Ricci-flat K\"{a}hler manifold with finite topology, then $Y=X\setminus D$ where $X$ is a compact
K\"{a}hler orbifold and $D$ supports $-K_X$.  Another motivation is the construction of complete Ricci-flat
K\"{a}hler metrics which are asymptotic to the K\"{a}hler cone $C(S)$ of a Sasaki-Einstein manifold
$S$.  Some explicit examples are given in~\cite{MS1} of Ricci-flat K\"{a}hler metrics on resolutions of Ricci-flat
K\"{a}hler cones $C(S)$.  The author has considered the general existence problem in~\cite{vC2}, and given further families
of examples in~\cite{vC3}.  This case is somewhat easier than that considered in this article.  In particular, the technical assumption
(\ref{cond}) is not needed.  These asymptotically conical Calabi-Yau metrics are of interest in the AdS/CFT correspondence~\cite{MS2,MS3}.
Here one considers a stack of D-branes at the singularity of the Calabi-Yau cone, and it is useful to work with Calabi-Yau resolutions
of this singularity.

In fact, one of the motivations of this work is that the problem of finding a Ricci-flat K\"{a}hler metric on an affine
variety is complimentary to the problem of finding such a metric on a resolution of a singularity.
There are two ways of smoothing a Ricci-flat K\"{a}hler cone $X=C(S)$.  One can possibly find a crepant resolution
$\pi:\check{Y}\rightarrow X$.  Or one may be able to analytically deform $X$ to a smooth variety $\hat{Y}$.
When both of these exist, this is an example of a \emph{geometric transition} which are of interest to physicists and mathematicians
studying Calabi-Yau manifolds~\cite{Ros}.  The familiar example is that of the quadric cone $X=\{(U,V,Y,Z)\in\C^4 :UV-XY=0\}$ which
is known to admit a small resolution $\check{Y}$ and a smoothing $\hat{Y}$ where all three of these manifolds are known to admit Ricci-flat
K\"{a}hler metrics~\cite{CanOs}.

In Section~\ref{sec:examp} an example is given where $X$ is the cone over $\cps_{(2)}^2$, the two-points blow-up.  The total space
of the canonical bundle of $\cps_{(2)}^2$ gives a resolution $\check{Y}$ of $X$, and $X$ can also be deformed to a smooth affine variety
$\hat{Y} =\cps^3_{(1)} \setminus D$ where $D\cong \cps_{(2)}^2$.  Theorem~\ref{thm:main} shows that both $\check{Y}$ and $\hat{Y}$
have Ricci-flat K\"{a}hler metrics.  This example is simplest case in Corollary~\ref{cor:main} which gives a new Ricci-flat manifold.

We then consider an infinite class of three dimensional toric K\"{a}hler cones possessing a symmetry which were previously considered by
the author in a different context~\cite{vC, vC1}.  They admit large families of smooth deformations and therefore provide an infinite
family of Ricci-flat manifolds which are smoothings of toric Gorenstein singularities.  One can construct many more
examples of three dimensional toric cones with non-trivial deformations by taking Minkowski sums of simple polygons.
Many of these, via Corollary~\ref{cor:main}, should give new examples of affine varieties with Ricci-flat K\"{a}hler metrics.

\subsection*{Notation}
We will denote line bundles in boldface, $\mathbf{L}, \mathbf{K}$, etc. While the corresponding divisor classes
are denoted, $L, K$, etc.  If $D$ is a divisor then $[D]$ denotes, depending on the context, either
the corresponding line bundle or the poincar\'{e} dual of the homology class of $D$.  The same notation will
be used for the analogous notions of V-bundles and Baily divisors on orbifolds (cf.~\cite{Ba}).  The total space
of the line bundle $\mathbf{L}$ minus the zero section will be denoted by $\mathbf{L}^\times$.

\subsection*{Acknowledgments}
This work was inspired by a conversation with James Sparks at the Sugadaira conference in 2007.
It is a pleasure to thank the organizers for the opportunity to participate in the conference.

\section{Sasaki manifolds}

\subsection{Introduction}\label{subsect:intro}
We review here some results from Sasakian geometry.  For more details see~\cite{BG},~\cite{FOW}, or the
comprehensive monograph~\cite{BG}.
\begin{defn}
A Riemannian manifold $(S,g)$ is Sasakian if the metric cone $(C(S),\ol{g})$,
$C(S)=S\times\R_+$ and $\ol{g}=dr^2 +r^2 g$, is a K\"{a}hler manifold.
\end{defn}

Thus $\dim_{\R} S=2m+1$, where $n=\dim_{\C} C(S)=m+1$.
Set $\tilde{\xi}=J(r\frac{\partial}{\partial r})$, then $\tilde{\xi}-iJ\tilde{\xi}$ is a holomorphic vector field
on $C(S)$.  The restriction $\xi$ of $\tilde{\xi}$ to $S=\{r=1\}\subset C(S)$
is the \emph{Reeb vector field} of $S$, which is a Killing vector field.
If the orbits of $\xi$ close, then it defines a locally free $U(1)$-action on $S$ and the Sasaki
structure is said to be \emph{quasi-regular}.  If the $U(1)$-action is free, then the Sasaki
structure is \emph{regular}.  If the orbits of $\xi$ do not close, then the Sasaki structure is \emph{irregular}.
In this case the closure of the one parameter subgroup of the isometry group generated by $\xi$ is a
torus $T^k$ and $T^k \subseteq\Aut(S)$, the automorphism group of the Sasaki structure.
We say that the Sasaki structure has \emph{rank} k, $\rank(S)=k$.

Let $\eta$ be the dual 1-form to $\xi$ with respect to $g$.  Then
\begin{equation}\label{eq:cont}
\eta = (2d^c \log r)|_{r=1},
\end{equation}
where $d^c=\frac{1}{2}i(\ol{\partial}-\partial)$.  Let $D=\ker\eta$.  Then $d\eta$ in non-degenerate on $D$
and $\eta$ is a contact form on $S$.  Furthermore, we have
\begin{equation}
d\eta(X,Y) =2g(\Phi X,Y), \quad\text{for }X,Y\in D_x,\ x\in S,
\end{equation}
where $\Phi$ is given by the restriction of the complex structure $J$ on $C(S)$ to $D_x$ and $\Phi\xi=0$.  Thus
$(D,J)$ is a strictly pseudo-convex CR structure on $S$.
We will denote the Sasaki structure on $S$ by $(g,\xi,\eta,\Phi)$.
It follows from (\ref{eq:cont}) that the K\"{a}hler form of $(C(S),\ol{g})$ is
\begin{equation}
 \omega=\frac{1}{2}d(r^2 \eta)=\frac{1}{2}dd^c r^2.
\end{equation}
Thus $\frac{1}{2}r^2$ is a K\"{a}hler potential for $\omega$.

There is a 1-dimensional foliation $\mathscr{F}_\xi$ generated by the Reed vector field $\xi$.  Since the leaf
space is identical with that generated by $\tilde{\xi}-iJ\tilde{\xi}$ on $C(S)$, $\mathscr{F}_\xi$ has
a natural transverse holomorphic structure.  And $\omega^T =\frac{1}{2}d\eta$ defines a K\"{a}hler form on the
leaf space.

We will consider deformations of the transverse K\"{a}hler structure.  Let $\phi\in C^\infty_B(S)$ be a smooth basic
function, meaning $\xi\contr d\phi=0$.  Then set
\begin{equation}\label{eq:trans}
\tilde{\eta} =\eta +2d^c_B \phi.
\end{equation}
Then
\[ d\tilde{\eta} =d\eta +2d_B d^c_B \phi =d\eta +2i\partial_B \ol{\partial}_B \phi. \]
For sufficiently small $\phi$, $\tilde{\eta}$ is a non-degenerate contact form in that $\tilde{\eta}\wedge d\tilde{\eta}^m$
is nowhere zero.  Then we have a new Sasaki structure on $S$ with the same Reeb vector field $\xi$, transverse
holomorphic structure on $\mathscr{F}_\xi$, and holomorphic structure on $C(S)$.  This Sasaki structure has
transverse K\"{a}hler form $\tilde{\omega}^T=\omega^T +i\partial_B \ol{\partial}_B \phi$.  One can show~\cite{FOW}
that if
\[\tilde{r} =r\exp{\phi},\]
then $\frac{1}{2}\tilde{r}^2$ is the K\"{a}hler potential of the new K\"{a}hler structure on $C(S)$.

\begin{prop}\label{prop:ricci}
Let $(S,g)$ be a $2m+1$-dimensional Sasaki manifold.  Then the following are equivalent.
\begin{thmlist}
 \item $(S,g)$ is Sasaki-Einstein with the Einstein constant being necessarily $2m$.

 \item $(C(S),\ol{g})$ is a Ricci-flat K\"{a}hler.

 \item The K\"{a}hler structure on the leaf space of $\mathscr{F}_\xi$ is K\"{a}hler-Einstein with Einstein constant $2m+2$.
\end{thmlist}
\end{prop}
This follows from elementary computations.  In particular, the equivalence of (i) and (iii) follows from
\begin{equation}\label{eq:ricci}
 \ric_g(\tilde{X},\tilde{Y})=(\Ric^T -2g^T)(X,Y),
\end{equation}
where $\tilde{X},\tilde{Y}\in D$ are lifts of $X,Y$ in the local leaf space.

We will make use of a slight generalization of the Sasaki-Einstein condition.
\begin{defn}
 A Sasaki manifold $(S,g)$ is $\eta$-Einstein if there are constants $\lambda$ and $\nu$ with
 \[ \Ric =\lambda g +\nu\eta\otimes\eta.\]
\end{defn}

We have $\lambda +\nu=2m$ as $\Ric(\xi,\xi)=2m$.  In fact, this condition is equivalent to the transverse K\"{a}hler-Einstein
condition $\Ric^T =\kappa\omega^T$, since this implies by the the same argument that proves Proposition~\ref{prop:ricci}
that
\begin{equation}
 \Ric =(\kappa -2)g +(2m+2-\kappa)\eta\otimes\eta,
\end{equation}
and conversely.

Given a Sasaki structure we can perform a $D$-homothetic transformation to get a new Sasaki structure.  For $a>0$ set
\begin{gather}
 \eta'=a\eta,\quad \xi'=\frac{1}{a}\xi,\\
 g'=ag^T +a^2\eta\otimes\eta =ag+(a^2-a)\eta\otimes\eta.\\
\end{gather}
Then $(g',\xi',\eta',\Phi)$ is a Sasaki structure with the same holomorphic structure on $C(S)$, and with
$r'=r^a$.

Suppose that $g$ is $\eta$-Einstein with $\Ric_g =\lambda g +\nu\eta\otimes\eta$.   A simple computation involving
(\ref{eq:ricci}), $\Ric'^T =\Ric^T$ and $\Ric_{g'}(\xi',\xi')=2m$ shows that the $D$-homothetic transformation gives
an $\eta$-Einstein Sasaki structure with
\begin{equation}
 \Ric_{g'} =\lambda'g'+(2m-\lambda')\eta\otimes\eta,  \text{ with }\lambda'=\frac{\lambda+2-2a}{a}.
\end{equation}
If $g$ is $\eta$-Einstein with $\lambda>-2$, then a $D$-homothetic transformation with $a=\frac{\lambda+2}{2m+2}$ gives a
Sasaki-Einstein metric $g'$.  Thus any Sasaki structure which is transversely K\"{a}hler-Einstein
$\Ric^T =\kappa\omega^T$ with $\kappa>0$ has a $D$-homothetic transformation to a Sasaki-Einstein structure.

\begin{prop}\label{prop:CY-cond}
The following necessary conditions for $S$ to admit a deformation of the transverse K\"{a}hler structure to a Sasaki-Einstein
metric are equivalent.
\begin{thmlist}
 \item $c_1^B =a[d\eta]$ for some positive constant $a$.

 \item $c_1^B >0$, i.e. represented by a positive $(1,1)$-form, and $c_1(D)=0$.

 \item For some positive integer $\ell>0$, the $\ell$-th power of the canonical line bundle \label{prop:CY-cond-iii}
 $\mathbf{K}^{\ell}_{C(S) }$ admits a nowhere vanishing section $\Omega$ with
 $\mathcal{L}_\xi \Omega =i(m+1)\Omega$.
\end{thmlist}
\end{prop}
If (\ref{prop:CY-cond-iii}) is satisfied then the singularity $X=C(S)\cup\{o\}$ is $\ell$-Gorenstein, meaning that $\mathbf{K}^{\ell}_{C(S)}$
is trivial.  Though, the condition of Proposition~\ref{prop:CY-cond} is stronger.  In fact the isolated singularity of $X$ is
rational (cf.~\cite{Bur}).  If $\ell=1$ we will say that $X$ is Gorenstein.

\begin{proof}
 Let $\rho$ denote the Ricci form of $(C(S),\ol{g})$, then easy computation shows that
 \begin{equation}\label{eq:ricci-cone}
  \rho =\rho^T -(2m+2)\frac{1}{2}d\eta.
 \end{equation}
If (i) is satisfied, there is a $D$-homothety so that $[\rho^T]=(2m+2)[\frac{1}{2}d\eta]$ as basic classes.
Thus there exists a smooth function $h$ with $\xi h=0=r\frac{\partial}{\partial r}h$ and
\begin{equation}
 \rho =i\partial\ol{\partial}h.
\end{equation}
This implies that $e^h \frac{\omega^{m+1}}{(m+1)!}$, where $\omega$ is the K\"{a}hler form of $\ol{g}$, defines
a flat metric $|\cdot|$ on $\mathbf{K}_{C(S)}$.  Parallel translation defines a multi-valued section which defines
a holomorphic section $\Omega$ of $\mathbf{K}^{l}_{C(S)}$ for some integer $l>0$ with
$|\Omega |=1$.  Then we have
\begin{equation}\label{eq:hol-form}
 \left(\frac{i}{2}\right)^{m+1}(-1)^{\frac{m(m+1)}{2}}\Omega\wedge\ol{\Omega} =e^h\frac{1}{(m+1)!}\omega^{m+1}.
\end{equation}
  From the invariance of $h$ and the fact that $\omega$ is homogeneous of degree 2, we see that
  $\mathcal{L}_{r\frac{\partial}{\partial r}}\Omega=(m+1)\Omega$.

  The equivalence of (i) and (ii) is easy (cf.~\cite{FOW} Proposition 4.3).
\end{proof}

\begin{xpl}\label{xpl:standard}
This is the most elementary construction of Sasaki manifolds.  Let $\mathbf{L}$ be a negative line bundle over
a complex manifold, or orbifold, $M$.  Then $\mathbf{L}$ has a hermitian metric $h$ with
$\omega_M =i\partial\ol{\partial}\log h$ a K\"{a}hler form.  Set $r_0^2 =h|z|^2$, then $\frac{r_0^2}{2}$ defines the
K\"{a}hler potential of a K\"{a}hler cone structure on $C(S)=\mathbf{L}^\times$, $\mathbf{L}$ minus the zero section,
with K\"{a}hler form $\omega =\frac{1}{2}\partial\ol{\partial}r_0^2$.
Thus $S=\{r_0 =1\}\subset\mathbf{L}$ has a natural Sasaki structure.
Note that $\eta$ is a, real valued, connection on the $S^1$ bundle $S$ and
$\omega^T =\frac{1}{2}d\eta=\omega_M$.  We call this Sasaki structure on $S=\{r_0 =1\}\subset\mathbf{L}$ the
\emph{standard} Sasaki structure.  This Sasaki structure is well defined up to deformations of the transverse K\"{a}hler
structure and $D$-homothetic transformations.

In particular if $\mathbf{L}= \mathbf{K}_M$, where $\mathbf{K}_M$ denotes the orbifold canonical bundle of $M$ if $M$ is an orbifold,
then this Sasaki structure on $S\subset\mathbf{L}^\times$ satisfies the conditions of Proposition~\ref{prop:CY-cond}.
\end{xpl}

We will make use of the notion of the Sasaki cone of a Sasaki structure~\cite{BGS1,BGS2}.
Given a Sasaki structure $(g,\xi,\eta,\Phi)$ with $D=\ker\eta$ and $\Phi|_D =J$, we consider the following set.
\begin{equation}
 \mathcal{S}(D,J)=\left\{
\begin{array}{c}
 S=(g,\xi,\eta,\Phi): S\text{ a Sasaki structure}\\
(\ker\eta, \Phi|_{\ker\eta})=(D,J)\\
\end{array}\right\}
\end{equation}
Thus $\mathcal{S}(D,J)$ is the set of Sasaki structures with underlying strictly pseudo-convex CR-structure $(D,J)$.

Let $\mathfrak{cr}(D,J)$ be the Lie algebra of the group $\mathfrak{CR}(D,J)$ of CR automorphisms $(D,J)$.
We say a vector field $X\in\mathfrak{cr}(D,J)$ is positive if $\eta(X)>0$ for any $(g,\xi,\eta,\Phi)\in\mathcal{S}(D,J)$.
If $\mathfrak{cr}^+(D,J)$ denotes the subset of positive vector fields, then it is easy to see that
\begin{equation}
 \begin{array}{ccc}
  \mathcal{S}(D,J) & \rightarrow & \mathfrak{cr}^+ (D,J)\\
  (g,\xi,\eta,\Phi) & \mapsto & \xi\\
 \end{array}
\end{equation}
is a bijection.  This leads to the following.
\begin{defn}
 Let $(D,J)$ be a CR structure associated to a Sasaki structure, then we have the \emph{Sasaki cone}
\[\kappa(D,J) = \mathcal{S}(D,J)/\mathfrak{CR}(D,J) \]
\end{defn}
The isotropy subgroup of an element $S$ is precisely the automorphism group
of the Sasaki structure, $\Aut(S)\subseteq\mathfrak{CR}(D,J)$.
\begin{thm}[\cite{BGS2}]\label{thm:CR-aut}
 Let $(D,J)$ be the CR structure associated to a Sasaki structure on a compact manifold.  Then the Lie algebra
$\mathfrak{cr}(D,J)$ decomposes as $\mathfrak{cr}(D,J)=\mathfrak{t}\oplus\mathfrak{p}$ where $\mathfrak{t}$ is the Lie
algebra of a maximal torus $T$ of dimension $k$, $1\leq k\leq m+1$, and $\mathfrak{p}$ is a completely reducible $T$-module.
Furthermore, every $X\in\mathfrak{cr}^+(D,J)$ is conjugate to a positive element of $\mathfrak{t}$.
\end{thm}

Considering Theorem~\ref{thm:CR-aut} let us fix a maximal torus $T$ of a maximal compact subgroup $G\subseteq\mathfrak{CR}(D,J)$
with Weyl group $\mathcal{W}$.  Let $\mathfrak{t}^+ =\mathfrak{t}\cap\mathfrak{cr}^+(D,J)$ denote the subset of positive elements.
Then we have
\begin{equation}
 \kappa(D,J) =\mathfrak{t}^+/\mathcal{W}.
\end{equation}

In practice we will alter a Sasaki structure $S=(g,\xi,\eta,\Phi)$ by varying in $\kappa(D,J)$ to $S'=(g',\xi',\eta', \Phi)$ and
then making a transverse K\"{a}hler deformation as in (\ref{eq:trans}) to $\tilde{S}=(\tilde{g},\xi',\tilde{\eta},\tilde{\Phi})$.
This is awkward to formalize since if $\phi\in C_B^{\infty}$ is not $T$ invariant, then the Sasaki cone of $\tilde{S}$ will have
dimension less than $k$.  So one generally fixes a torus $T^k$ and considers $T^k$-invariant Sasaki structures on $S$.
The case $k=m+1$ is well understood and is considered in the next section.

The proof of the following is easy.
\begin{prop}
 Suppose $(g,\xi,\eta,\Phi)$ is a Sasaki structure on $S$ satisfying Proposition~\ref{prop:CY-cond}, then every Sasaki structure
in $\kappa(D,J)$ satisfies Proposition~\ref{prop:CY-cond}.
\end{prop}

\subsection{Toric Sasaki-Einstein manifolds}\label{subsect:toric}

In this section we recall the basics of toric Sasaki manifolds.  Much of what follows can be found in~\cite{MSY} or~\cite{FOW}.

\begin{defn}
 A Sasaki manifold $(S,g,\xi,\eta,\Phi)$ of dimension $2m+1$ is \emph{toric} if there is an effective action of an
 $m+1$-dimensional torus $T=T^{m+1}$ preserving the Sasaki structure such that $\xi$ is an element of the Lie algebra
 $\mathfrak{t}$ of $T$.  Equivalently, a toric Sasaki manifold is a Sasaki manifold $S$ whose K\"{a}hler cone
 $C(S)$ is a toric K\"{a}hler manifold.
\end{defn}

We have an effective holomorphic action of $T_\C \cong (\C^*)^{m+1}$ on $C(S)$ whose restriction to
$T\subset T_\C$ preserves the K\"{a}hler form $\omega =d(\frac{1}{2}r^2 \eta)$.  So there is a moment map
\begin{equation}\label{eq:moment-map}
\begin{gathered}
 \mu: C(S) \longrightarrow \mathfrak{t}^* \\
 \langle \mu(x),X\rangle = \frac{1}{2}r^2\eta(X_S (x)),
\end{gathered}
\end{equation}
where $X_S$ denotes the vector field on $C(S)$ induced by $X\in\mathfrak{t}$.  We have the
moment cone defined by
\begin{equation}
 \mathcal{C}(\mu) :=\mu(C(S)) \cup \{0\},
\end{equation}
which from~\cite{Ler} is a strictly convex rational polyhedral cone.  Recall that this means that there are vectors
$\lambda_i,i=1,\ldots,d$ in the integral lattice $\Z_T =\ker\{\exp:\mathfrak{t}\rightarrow T\}$ such that
\begin{equation}\label{eq:moment-cone}
 \mathcal{C}(\mu) =\bigcap_{j=1}^{d} \{y\in\mathfrak{t}^* : \langle\lambda_j,y\rangle\geq 0\}.
\end{equation}
The condition that $\mathcal{C}(\mu)$ is strictly convex means that it is not contained in any linear subspace of $\mathfrak{t}^*$.
It is cone over a finite polytope.
We assume that the set of vectors $\{\lambda_j\}$ is minimal in that removing one changes the set defined by
(\ref{eq:moment-cone}).  And we furthermore assume that the vectors $\lambda_j$ are primitive, meaning that
$\lambda_j$ cannot be written as $p\tilde{\lambda}_j$ for $p\in\Z$ and $\tilde{\lambda}_j\in\Z_T$.

Let $\inter\mathcal{C}(\mu)$ denote the interior of $\mathcal{C}(\mu)$.  Then the action of $T$ on
$\mu^{-1}(\inter\mathcal{C}(\mu))$ is free and is a Lagrangian torus fibration over $\inter\mathcal{C}(\mu)$.
There is a condition on the $\{\lambda_j\}$ for $S$ to be a smooth manifold.  Each face
$\mathcal{F}\subset\mathcal{C}(\mu)$ is the intersection of a number of facets
$\{y\in\mathfrak{t}^* :l_j(y)=\lambda_j \cdot y =0\}$.
For $S$ to be smooth each face $\mathcal{F}$ must be the intersection of $\codim\mathcal{F}$ facets.
Let $\lambda_{j_1},\ldots,\lambda_{j_a}$ be the corresponding
collection of normal vectors in $\{\lambda_j\}$, where $a$ is the codimension of $\mathcal{F}$.  Then the cone
$\mathcal{C}(\mu)$ is smooth if and only if
\begin{equation}\label{eq:smooth}
 \left\{ \sum_{k=1}^{a} \nu_{k}\lambda_{j_k} :\nu_k \in\R\right\}\cap\Z_T =\left\{\sum_{k=1}^{a} \nu_{k}\lambda_{j_k} :\nu_k \in\Z\right\}
\end{equation}
for all faces $\mathcal{F}$.

Note that $\mu(S)=\{y\in\mathcal{C}(\mu) : y(\xi)=\frac{1}{2}\}$.  The hyperplane $\{y\in\mathfrak{t}^* : y(\xi)=\frac{1}{2}\}$
is the \emph{characteristic hyperplane} of the Sasaki structure.
Consider the dual cone to $\mathcal{C}(\mu)$
\begin{equation}\label{eq:dual-cone}
 \mathcal{C}(\mu)^* =\{\tilde{x}\in\mathfrak{t} :\tilde{x}\cdot y\geq 0\text{ for all }y\in\mathcal{C}(\mu)\},
\end{equation}
which is also a strictly convex rational polyhedral cone by Farkas' theorem.  Then $\xi$ is in the interior of $\mathcal{C}(\mu)^*$.
Let $\frac{\partial}{\partial\phi_i},i=1,\ldots,m+1,$ be a basis of $\mathfrak{t}$ in $\Z_T$.
Then we have the identification $\mathfrak{t}^*\cong\mathfrak{t}\cong\R^{m+1}$ and write
\[ \lambda_j =(\lambda_j^1, \ldots,\lambda_j^{m+1}),\quad \xi=(\xi^1,\ldots,\xi^{m+1}). \]
If we set
\begin{equation}\label{eq:symp-coord}
y_i =\mu(x)\left(\frac{\partial}{\partial\phi_i}\right)\quad ,i=1,\ldots,m+1,
\end{equation}
then we have symplectic coordinates $(y,\phi)$ on $\mu^{-1}(\inter\mathcal{C}(\mu))\cong \inter\mathcal{C}(\mu)\times T^{m+1}$.
In these coordinates the symplectic form is
\begin{equation}
 \omega =\sum_{i=1}^{m+1} dy_i \wedge d\phi_i.
\end{equation}
The K\"{a}hler metric can be seen as in~\cite{Abr} to be of the form
\begin{equation}
 g=\sum_{ij} G_{ij}dy_i dy_j + G^{ij}d\phi_i d\phi_j,
\end{equation}
where $G^{ij}$ is the inverse matrix to $G_{ij}(y)$, and the complex structure is
\begin{equation}\label{eq:comp-st}
 \mathcal{I} =\left\lgroup
\begin{matrix}
 0 & -G^{ij} \\
 G_{ij} & 0 \\
\end{matrix}\right\rgroup
\end{equation}
in the coordinates $(y,\phi)$.  The integrability of $\mathcal{I}$ is $G_{ij,k} =G_{ik,j}$.  Thus
\begin{equation}
 G_{ij}=G_{,ij} :=\frac{\partial^2 G}{\partial y_i \partial y_j},
\end{equation}
for some strictly convex function $G(y)$ on $\inter\mathcal{C}(\mu)$.  We call $G$ the symplectic potential of the
K\"{a}hler metric.

One can construct a canonical K\"{a}hler structure on the cone $X=C(S)$, with a fixed holomorphic structure,
via a simple K\"{a}hler reduction of $\C^d$ (cf.~\cite{Gui} and~\cite{MSY}).  The symplectic potential of the
canonical K\"{a}hler metric is
\begin{equation}
 G^{can} =\frac{1}{2}\sum_{i=1}^{d}l_i (y)\log l_i (y).
\end{equation}
Let
\[ G_\xi =\frac{1}{2}l_{\xi}(y)\log l_{\xi} -\frac{1}{2}l_{\infty}(y)\log l_{\infty}(y),\]
where
\[ l_{\xi}(y)=\xi\cdot y, \text{ and } l_{\infty}(y)=\sum_{i=1}^d \lambda_i \cdot y.\]
Then
\begin{equation}\label{eq:can-pot}
 G_\xi ^{can} =G^{can} + G_\xi,
\end{equation}
defines a symplectic potential of a K\"{a}hler metric on $C(S)$ with induced Reeb vector field $\xi$.
To see this write
\begin{equation}
 \xi =\sum_{i=1}^{m+1} \xi^i \frac{\partial}{\partial\phi_i},
\end{equation}
and note that the Euler vector field is
\begin{equation}
 r\frac{\partial}{\partial r} =2\sum_{i=1}^{m+1}y_i\frac{\partial}{\partial y_i}.
\end{equation}
Thus we have
\begin{equation}\label{eq:sym-reeb}
 \xi^i =\sum_{j=1}^{m+1} 2G_{ij}y_j.
\end{equation}
Computing from (\ref{eq:can-pot}),
\begin{equation}\label{eq:can-pot-der}
 \left(G_\xi ^{can} \right)_{ij} =\frac{1}{2}\sum_{k=1}^d \frac{\lambda_k^i \lambda_k^j}{l_k (y)} +\frac{1}{2}\frac{\xi^i \xi^j}{l_\xi (y)}
 -\frac{1}{2}\frac{\sum_{k=1}^d \lambda_k^i \sum_{k=1}^d \lambda_k^j}{l_\infty (y)},
\end{equation}
and (\ref{eq:sym-reeb}) follows by direct computation.

The general symplectic potential is of the form
\begin{equation}
 G =G^{can} + G_\xi +g,
\end{equation}
where $g$ is a smooth homogeneous degree one function on $\mathcal{C}$ such that $G$ is strictly convex.

Note that the complex structure on $X=C(S)$ is determined up to biholomorphism by the associated moment
polyhedral cone $\mathcal{C}(\mu)$ (cf.~\cite{Abr} Proposition A.1).  The following follows easily from this discussion.
\begin{prop}
 Let $S$ be a compact toric Sasaki manifold and $C(S)$ its K\"{a}hler cone.  For any
 $\xi\in\inter\mathcal{C}(\mu)^*$ there exists a toric K\"{a}hler cone metric, and associated Sasaki structure on $S$,
 with Reeb vector field $\xi$.  And any other such structure is a transverse K\"{a}hler deformation, i.e.
 $\tilde{\eta} =\eta +2d^c \phi$, for a basic function $\phi$.
\end{prop}

Consider now the holomorphic picture of $C(S)$.  Since $C(S)$ is a toric variety
$(\C^*)^{m+1}\cong\mu^{-1}(\inter\mathcal{C})\subset C(S)$ is an dense orbit.  We introduce logarithmic
coordinates $(z_1,\ldots,z_{m+1}) =(x_1 +i\phi_1,\ldots, x_{m+1}+i\phi_{m+1})$ on
$\C^{m+1}/{2\pi i\Z^{m+1}}\cong (\C^*)^{m+1}\cong\mu^{-1}(\inter\mathcal{C})\subset C(S)$, i.e.
$x_j +i\phi_j =\log w_j$ if $w_j,j=1,\ldots,m+1$, are the usual coordinates on $(\C^*)^{m+1}$.
The K\"{a}hler form can be written as
\begin{equation}
 \omega =i\partial\ol{\partial}F,
\end{equation}
where $F$ is a strictly convex function of $(x_1,\ldots,x_{m+1})$.  One can check that
\begin{equation}
 F_{ij}(x) =G^{ij}(y),
\end{equation}
where $\mu=y=\frac{\partial F}{\partial x}$ is the moment map.  Furthermore, one can show $x=\frac{\partial G}{\partial y}$, and
the K\"{a}hler and symplectic potentials are related by the Legendre transform
\begin{equation}
 F(x)= \sum_{i=1}^{m+1} x_i \cdot y_i -G(y).
\end{equation}
It follows from equation (\ref{eq:symp-coord}) defining symplectic coordinates that
\begin{equation}
 F(x)= l_\xi (y) =\frac{r^2}{2}.
\end{equation}

We now consider the conditions in Proposition~\ref{prop:CY-cond} more closely in the toric case.
So suppose the Sasaki structure satisfies Proposition~\ref{prop:CY-cond}, thus we may assume
$c_1^B =(2m+2)[\omega^T]$.  Then equation~\ref{eq:ricci-cone} implies that
\begin{equation}
 \rho =-i\partial\ol{\partial}\log\det(F_{ij})=i\partial\ol{\partial}h,
\end{equation}
with $\xi h=0=r\frac{\partial}{\partial r}h$, and we may assume $h$ is $T^{m+1}$-invariant.
Since a $T^{m+1}$-invariant pluriharmonic function is an affine function, we have
constants $\gamma_1,\ldots,\gamma_{m+1}\in\R$ so that
\begin{equation}\label{eq:hol-CY}
 \log\det(F_{ij})=-2\sum_{i=1}^{m+1}\gamma_i x_i -h.
\end{equation}
In symplectic coordinates we have
\begin{equation}\label{eq:symp-CY}
 \det(G_{ij})=\exp(2\sum_{i=1}^{m+1}\gamma_i G_i +h).
\end{equation}
Then from (\ref{eq:can-pot}) one computes the right hand side to get
\begin{equation}
 \det(G_{ij})=\prod_{k=1}^d \left(\frac{l_k (y)}{l_\infty (y)}\right)^{(\gamma,\lambda_k)} (l_\xi (y))^{-(m+1)} \exp(h),
\end{equation}
And from (\ref{eq:can-pot-der}) we compute the left hand side of (\ref{eq:symp-CY})
\begin{equation}
 \det(G_{ij})=\prod_{k=1}^d(l_k (y))^{-1} f(y),
\end{equation}
where $f$ is a smooth function on $\mathcal{C}(\mu)$.  Thus $(\gamma,\lambda_k)=-1$, for $k=1,\ldots,d$.
Since $\mathcal{C}(\mu)^*$ is strictly convex, $\gamma$ is a uniquely determined element of $\mathfrak{t}^*$.

Applying $\sum_{j=1}^{m+1}y_j \frac{\partial}{\partial y_j}$ to (\ref{eq:symp-CY}) and noting that
$\det(G_{ij})$ is homogeneous of degree $-(m+1)$ we get
\begin{equation}\label{eq:reeb-const}
 (\gamma,\xi)=-(m+1).
\end{equation}

As in Proposition~\ref{prop:CY-cond} $e^h \det(F_{ij})$ defines a flat metric $\|\cdot\|$ on
$\mathbf{K}_{C(S)}$.  Consider the $(m+1,0)$-form
\[ \Omega = e^{i\theta}e^{\frac{h}{2}}\det(F_{ij})^{\frac{1}{2}} dz_1 \wedge\cdots\wedge dz_{m+1}.\]
From equation (\ref{eq:hol-CY}) we have
\[ \Omega =e^{i\theta}\exp(-\sum_{j=1}^{m+1}\gamma_j x_j)dz_1 \wedge\cdots\wedge dz_{m+1}.\]
If we set $\theta=-\sum_{j=1}^{m+1}\gamma_j \phi_j$, then
\begin{equation}\label{eq:hol-form-toric}
 \Omega=e^{-\sum_{j=1}^{m+1}\gamma_j z_j}dz_1 \wedge\cdots\wedge dz_{m+1}
\end{equation}
is clearly holomorphic on $U=\mu^{-1}(\inter\mathcal{C})$.  When $\gamma$ is not integral, then we take $\ell\in\Z_+$
such that $\ell\gamma$ is a primitive element of $\Z_T^* \cong\Z^{m+1}$.  Then
$\Omega^{\otimes\ell}$ is a holomorphic section of $\mathbf{K}^{\ell}_{C(S)}|_U$ which extends to
a holomorphic section of $\mathbf{K}^{\ell}_{C(S)}$ as $\|\Omega\|=1$.

It follows from (\ref{eq:hol-form-toric}) that
\begin{equation}
 \mathcal{L}_\xi \Omega =-i(\gamma,\xi)\Omega =i(m+1)\Omega.
\end{equation}
And note that we have equation (\ref{eq:hol-form}) from (\ref{eq:hol-CY}) and (\ref{eq:hol-form-toric}).

\begin{prop}\label{prop:CY-cond-toric}
Let $S$ be a compact toric Sasaki manifold.  Then the conditions of Proposition~\ref{prop:CY-cond}
are equivalent to the existence of $\gamma\in\mathfrak{t}^*$ such that
\begin{thmlist}
 \item $(\gamma,\lambda_k)=-1$, for $k=1,\ldots,d$,

 \item $(\gamma,\xi)=-(m+1)$, and

 \item there exists $\ell\in\Z_+$ such that $\ell\gamma\in\Z_T^* \cong\Z^{m+1}$
\end{thmlist}
Then (\ref{eq:hol-form-toric}) defines a nowhere vanishing section of $\mathbf{K}^{\ell}_{C(S)}$.
\end{prop}

In the toric case the conditions of Proposition~\ref{prop:CY-cond-toric} are equivalent to $X=C(S)\cup\{o\}$ being an $\ell$-Gorenstein singularity.

We will need the beautiful results of A. Futaki, H. Ono, and G. Wang on the existence of
Sasaki-Einstein metrics on toric Sasaki manifolds.
\begin{thm}[\cite{FOW,CFO}]\label{thm:FOW}
Suppose $S$ is a toric Sasaki manifold satisfying Proposition~\ref{prop:CY-cond-toric}.
Then we can deform the Sasaki structure by varying the Reeb vector field and then performing
a transverse K\"{a}hler deformation to a Sasaki-Einstein metric.  The Reeb vector field and transverse
K\"{a}hler deformation are unique up to isomorphism.
\end{thm}

In~\cite{FOW} a more general result is proved.  It is proved that a compact toric Sasaki
manifold satisfying Proposition~\ref{prop:CY-cond-toric} has a transverse K\"{a}hler deformation
to a Sasaki structure satisfying the transverse K\"{a}hler Ricci soliton equation:
\[ \rho^T -(2m+2)\omega^T =\mathcal{L}_X \omega^T \]
for some Hamiltonian holomorphic vector field $X$.  The analogous result for toric Fano manifolds was
proved in~\cite{WaZh}.  A transverse K\"{a}hler Ricci soliton becomes a transverse K\"{a}hler-Einstein
metric, i.e. $X=0$, if the Futaki invariant $f_1$ of the transverse K\"{a}hler structure vanishes.
The invariant $f_1$ depends only on the Reeb vector field $\xi$.  The next step is to use a
volume minimization argument due to Martelli-Sparks-Yau~\cite{MSY} to show there is a unique $\xi$
satisfying (\ref{eq:reeb-const}) for which $f_1$ vanishes.

\begin{xpl}\label{xpl:FOW}
Let $M=\cps^2_{(2)}$ be the two-points blow up.  And Let $S\subset\mathbf{K}_M$ be the
$U(1)$-subbundle of the canonical bundle.  Then the standard Sasaki structure on $S$
satisfies (1) of Proposition~\ref{prop:CY-cond}, and it is not difficult to show that $S$
is simply connected and is toric.  See Example~\ref{xpl:standard}.
But the automorphism group of $M$ is not reductive, thus $M$ does not admit a K\"{a}hler-Einstein metric
due to Y. Matsushima~\cite{Mat}.  Thus there is no Sasaki-Einstein structure with the usual Reeb vector field.
But by Theorem~\ref{thm:FOW} there is a Sasaki-Einstein structure with a different Reeb vector field.

The vectors defining the facets of $\mathcal{C}(\mu)$ are
\[\lambda_1 =(1,0,0),\ \lambda_2 =(1,0,1),\ \lambda_3 =(1,1,2),\ \lambda_4 =(1,2,1),\ \lambda_5 =(1,1,0).\]
The Reeb vector field of the toric Sasaki-Einstein metric on $S$ was calculated in~\cite{MSY} to be
\[\xi =\left(3, \frac{9}{16}(-1+\sqrt{33}), \frac{9}{16}(-1+\sqrt{33})\right).\]
One sees that the Sasaki structure is irregular with the closure of the generic orbit being a two torus.
\end{xpl}

\section{Approximating metric}

\subsection{The Calabi Ansatz}\label{subsect:calabi}
The Calabi ansatz constructs a complete Ricci-flat K\"{a}hler metric on the total space of the canonical
bundle $\mathbf{K}_M$ of a K\"{a}hler manifold $(M,\omega)$, provided $M$ admits a K\"{a}hler-Einstein metric.
This condition is equivalent, up to homothety, to the standard Sasaki structure on
$S\subset\mathbf{K}_M$ being Einstein, where $S=\{r=1\}$ with the Sasaki structure $(g,\xi,\eta,\Phi)$ with
$\frac{1}{2}d\eta =\omega$ and $\xi$  generating the $S^1$ action on $\mathbf{K}_M$.
We describe an extension of the Calabi ansatz, due to A. Futaki~\cite{Fut3}, to the case where $S$ admits a Sasaki-Einstein 
structure for a possibly different Reeb vector field $\tilde{\xi}$, with the same K\"{a}hler cone.

Suppose $M$ is a Fano manifold and $\mathbf{L}^p =\mathbf{K}_M$ for a positive integer $p$.
Suppose there is an $\eta$-Einstein Sasaki structure $(g,\xi,\eta,\Phi)$ on the $U(1)$-bundle $S$
associated to $\mathbf{L}$, where $\xi$ is a possibly different Reeb vector field to that in the standard
Sasaki structure of Example~\ref{xpl:standard}.  Thus
\begin{equation}
\rho^T =\kappa\omega^T,
\end{equation}
where we set $\kappa=2p$.  Set $t=\log r$.  The Calabi ansatz searches for a K\"{a}hler form on
$\mathbf{L}$ of the form
\begin{equation}
\omega_\phi =\omega^T +i\partial\ol{\partial}F(t),
\end{equation}
where $F(t)$ is a smooth function on $(t_1,t_2)\subset (-\infty,\infty)$.  Define a new variable and function
\begin{gather}
\tau=F'(t) \\
\phi(\tau)=F''(t).
\end{gather}
We must require $\phi(\tau)>0$ for $\omega$ to be positive.  Also assume that $F'$ maps $(t_1,t_2)$ onto
$(0,b)$.  Then the Calabi ansatz is
\begin{equation}\label{eq:cal-ansatz}
\begin{split}
\omega_\phi & = \omega^T +dd^c F(t) \\
            & = (1+\tau)\omega^T +\phi(\tau)i\partial t\wedge\ol{\partial}t \\
            & = (1+\tau)\omega^T +\phi(\tau)^{-1}i\partial\tau\wedge\ol{\partial}\tau \\
\end{split}
\end{equation}
which defines a K\"{a}hler metric on
\begin{equation}
C(S)_{(t_1,t_2)} =\{e^{t_1}<r< e^{t_2}\}\subseteq C(S)\subset\mathbf{L}.
\end{equation}

Direct computation gives the equations
\begin{equation}\label{eq:vol}
\omega_\phi^{m+1} = (1+\tau)^m (m+1)\phi(\tau)dt\wedge d^c t\wedge(\omega^T)^m,
\end{equation}
\begin{equation}
\begin{split}
\rho_\phi & = \rho^T -i\partial\ol{\partial}\log((1+\tau)^m\phi(\tau)) \\
          & = \kappa\omega^T -i\partial\ol{\partial}\log((1+\tau)^m\phi(\tau)), \\
\end{split}
\end{equation}\label{eq:scalar}
\begin{equation}
\begin{split}
\sigma_\phi & = \frac{\sigma^T}{1+\tau} -i\Delta_\phi \log((1+\tau)^m\phi(\tau)) \\
            & = \frac{m\kappa}{1+\tau} -i\Delta_\phi \log((1+\tau)^m\phi(\tau)). \\
\end{split}
\end{equation}

It will be useful to know the relation between the curvature tensors of $\omega_\phi$ and $\omega^T$.
Denote them respectively by $R^\phi$ and $R^T$.  Denote by
$\zeta= r\frac{\partial}{\partial r} -i\xi$ the holomorphic vector field given by the Sasaki structure.
Let $U,V,X,Y$ be complex vector fields which are horizontal with respect to the 1-form
$\frac{dr}{r}+i\eta$ dual to $\zeta$.  Then we have
\begin{equation}\label{eq:curvature}
\begin{split}
& R^\phi (U,\ol{V},X,\ol{Y})  =(1+\tau)R^T (U,\ol{V},X,\ol{Y}) +\phi(\omega^T (U,\ol{V})\omega^T (X,\ol{Y})-\omega^T (U,\ol{Y})\omega^T (X,\ol{V})),\\
& R^\phi (U,\ol{V},\zeta,\ol{\zeta})  =(\phi -(1+\tau)^{-1}\phi^2 )i\omega^T (U,\ol{V}), \\
& R^\phi (\zeta,\ol{\zeta},\zeta,\ol{\zeta})  =-\ddot{\phi}+\phi^{-1}\dot{\phi}^2, \\
\end{split}
\end{equation}
where dots in the last line denote the derivative with respect to $t$.

We now consider the case of constant scalar curvature.  Calculation gives
\begin{equation}
\sigma_\phi = \frac{m\kappa}{1+\tau} -\frac{1}{(1+\tau)^m}\frac{d^2}{d\tau^2}((1+\tau)^m\phi).
\end{equation}
Setting $\sigma_\phi =c$ we get the differential equation
\begin{equation}
(\phi(1+\tau)^m)'' =\left(\frac{m\kappa}{(1+\tau)} -c\right)(1+\tau)^m,
\end{equation}
with the solutions
\begin{equation}\label{eq:prof1}
\phi(\tau) = \frac{\kappa}{m+1}(1+\tau) -\frac{c}{(m+1)(m+2)}(1+\tau)^2 +\frac{c_1\tau +c_2}{(1+\tau )^m},
\end{equation}
with constants $c_1$ and $c_2$.

The function
\begin{equation}\label{eq:geod}
s(t)=\int_{\tau_0}^{\tau(t)}\frac{dx}{\sqrt{\phi(x)}}
\end{equation}
gives the geodesic length along the $t$-direction.
We are interested in metrics with a complete end at infinity.  The following is a consequence of (\ref{eq:geod}).
\begin{prop}\label{prop:total-sp}
Let $\omega_\phi$ be the K\"{a}hler form of the Calabi ansatz on $\mathbf{L}^\times$ for an $\eta$-Einstein
Sasaki manifold.  Suppose $\phi$ is defined on $(c,\infty)$ and for some $c\geq 0$.  Then $\omega_\phi$ defines a metric
with a complete non-compact end toward $\tau=\infty$ on $\mathbf{L}$ if and only if $\phi$ grows at most quadratically as $\tau\rightarrow\infty$.
\end{prop}

We now construct Ricci-flat metrics on a neighborhood of infinity on $\mathbf{L}$ with $\mathbf{L}^p =\mathbf{K}_M$,
where $p=\alpha -1$.  Thus $\kappa=2(\alpha-1)$.
The desired metric must be complete and have a pole of order $2\alpha$ at infinity.
Calculation gives
\begin{equation}
\begin{split}
\rho_\phi & =\kappa\omega^T -i\partial\ol{\partial}\log((1+\tau)^m\phi(\tau)) \\
          & =\left(\kappa-\frac{m\phi+(1+\tau)\phi'}{1+\tau}\right)\omega^T -\left(\left(\frac{m\phi}{1+\tau}\right)'+\phi''\right)\phi dt\wedge d^c t.\\
\end{split}
\end{equation}
Thus $\kappa-\frac{m\phi+(1+\tau)\phi'}{1+\tau}=0$ and $\left(\frac{m\phi}{1+\tau} +\phi'\right)'=0$.
Thus
\begin{equation}
\frac{m\phi}{1+\tau} +\phi' =\kappa.
\end{equation}
And solving this equation gives
\begin{equation}\label{eq:prof2}
 \phi(\tau)=\frac{\kappa}{m+1}(1+\tau) +\frac{a}{(1+\tau)^m}, \text{  for }a\in\R.
\end{equation}
Therefore $\omega_\phi$ is Ricci-flat and is complete toward infinity as $\phi(\tau)$ has less than quadratic growth.

Now solve $\frac{d\tau}{dt}=\phi(\tau)$ to get
\begin{equation}
 \tau=\left(c\frac{m+1}{\kappa}e^{\kappa t} -a\frac{m+1}{\kappa}\right)^{\frac{1}{m+1}}-1,\text{ for }c>0.
\end{equation}
After changing the constants $a,c$ we have
\begin{gather}
 \tau = F'(t)=(ce^{\kappa t} +a)^{\frac{1}{m+1}} -1 \label{eq:prof3}\\
 \phi = F''(t)=\frac{c\kappa}{m+1}(ce^{\kappa t} +a)^{-\frac{m}{m+1}} e^{\kappa t}. \label{eq:prof4}
\end{gather}
It follows that equation~(\ref{eq:vol}) becomes
\begin{equation}\label{eq:vol-part}
 \omega_\phi^{m+1} =c\kappa e^{\kappa t}d t\wedge d^c t\wedge(\omega^T)^m.
\end{equation}
Notice that
\begin{equation}\label{eq:potent}
 G=G(t)=\int_{t_0}^{t}(ce^{\kappa s} +a)^{\frac{1}{m+1}}ds
\end{equation}
is a K\"{a}hler potential for $\omega_\phi$, i.e. $\omega_\phi =i\partial\ol{\partial}G$.

Thus (\ref{eq:cal-ansatz}) defines a Ricci-flat K\"{a}hler metric on $\mathbf{L}^\times$ for
the profile function (\ref{eq:prof4}) depending on constants $a\geq 0$ and $c>0$.
This two parameter family includes homotheties.  Note that for $a=0$ this is just a rescaling
of the Sasaki cone metric as is seen by changing $\omega_\phi$ to the coordinate
$r' =r^{\frac{\kappa}{2(m+1)}}$.

\subsection{approximating metric}

Let $\mathbf{L}=\mathbf{K}_M$ be the canonical bundle of a Fano manifold $M$ with $\pi:\mathbf{L}\rightarrow M$.
Then as in Example~\ref{xpl:standard}  there is a standard Sasaki structure on
$S=\{r_0 =1\}\subset\mathbf{L}$ with K\"{a}hler potential $r_0^2 =h|z|^2$
on $C(S)=\mathbf{L}^\times$ for $h$ a hermitian metric on $\mathbf{L}$.
Let $\Psi\in\Omega^{m,0}(\mathbf{L})$ be the tautologically defined holomorphic $m$-form on the total space of
$\mathbf{L}$, i.e. for $u\in\mathbf{L}$, $\Psi(u)=\pi^* u$.  Define a $(m+1,0)$-form
\begin{equation}
 \Omega = \left(\frac{dr}{r} +i\eta\right)\wedge\Psi.
\end{equation}
If $dz_1 \wedge\cdots\wedge dz_m$ is a local section giving fiber coordinate $w$, then
\begin{equation}
 \Omega =w \left(\frac{dr}{r} +i\eta\right)\wedge dz_1 \wedge\cdots\wedge dz_m.
\end{equation}
One easily checks that $d\Omega =0$, thus $\Omega$ is holomorphic.  Also, $\Omega$ has a pole of order 2 at $\infty$,
and $\mathcal{L}_\xi \Omega=i\Omega$.

We assume now that $M$ is a Fano orbifold.  And let $\mathbf{L}$ be a line bundle on $M$ with $\mathbf{L}^p =\mathbf{K}_M$.
Suppose there exists an $\eta$-Einstein Sasaki structure on the link $S$ with holomorphic cone $C(S)=\mathbf{L}^\times$
with K\"{a}hler potential $\frac{r^2}{2}$, Reeb vector field $\xi$ and contact form $\eta$, with
$\rho^T =\kappa\omega^T$, $\kappa =2p$, and which is compatible with $\Omega$ in that
$\mathcal{L}_\xi \Omega =ip\Omega$.

Consider the holomorphic map $\varpi:\mathbf{L}\overset{\sm \otimes^p}{\rightarrow}\mathbf{K}_M$.  Let
$\Omega' \in\Omega^{m+1,0}(\mathbf{K}_M)$ be the holomorphic form defined above.  Define $\Omega=\varpi^*\Omega'$.
Then $\Omega$ has a pole of order $p+1$ along $\infty$.  We have $\mathcal{L}_\xi \Omega=ip\Omega$, and
it is clear that this is the holomorphic form in Proposition~\ref{prop:CY-cond}.

Write $\ol{\omega}$ for $\omega_\phi$ defined in (\ref{eq:cal-ansatz}) using the profile $\phi$ defined
in equation (\ref{eq:prof3}).  Then $\ol{\omega}$ defines a Ricci-flat K\"{a}hler metric in a neighborhood of
$\infty$ on $\mathbf{L}$.  From (\ref{eq:vol-part}) we have
\begin{equation}\label{eq:volume}
\ol{\omega}^{m+1} =\frac{1}{2}c\kappa r^{\kappa -1} dr\wedge\eta\wedge(\omega^T)^m.
\end{equation}
Let $\omega =rdr\wedge\eta +r^2 \omega^T$ be the K\"{a}hler form of the
$\eta$-Einstein Sasaki structure.  We make a $D$-homothety as in Section~\ref{subsect:intro} with
$a=\frac{p}{m+1}$ to a Sasaki-Einstein structure on $\mathbf{L}^\times$ with $r'=r^{\frac{p}{m+1}}$,
$\eta'=\frac{p}{m+1}\eta$, and $\xi'=\frac{m+1}{p}\xi$.  Let $\omega'=r'dr'\wedge\eta'+r'^2 \omega'^T$ be the K\"{a}hler
form.  Then an easy computation gives
\begin{equation}\label{eq:vol-homo}
 (\omega')^{m+1} =a^{m+2}r^{(a-1)(2m+2)}\omega^{m+1}=\left(\frac{p}{m+1}\right)^{m+2}r^{2(p-m-1)}\omega^{m+1}.
\end{equation}
Now since $\omega'$ is a Ricci-flat K\"{a}hler form we have
\begin{equation}\label{eq:vol-Ricci}
 \left(\frac{i}{2}\right)^{m+1}(-1)^{\frac{m(m+1)}{2}}\Omega\wedge\ol{\Omega} =e^h\frac{1}{(m+1)!}(\omega')^{m+1},
\end{equation}
with $\partial\ol{\partial}h=0$.  Since $\mathcal{L}_{\xi'} \Omega=i(m+1)\Omega$, we have
$\xi'h=r'\frac{\partial h}{\partial r'} =0$, i.e. $h$ is basic.  Thus $h$ is constant.
And from (\ref{eq:volume}), (\ref{eq:vol-homo}), and (\ref{eq:vol-Ricci}) we have
\begin{equation}\label{eq:vol-Calab-Sasak}
\begin{split}
\ol{\omega}^{m+1} & =\frac{1}{2}c\kappa r^{2p -1} dr\wedge\eta\wedge(\omega^T)^m \\
          & = \frac{1}{2(m+1)} c\kappa r^{2(p-m-1)}\omega^{m+1} \\
          & = \ol{c}\Omega\wedge\ol{\Omega},\\
\end{split}
\end{equation}
where $\ol{c}$ is a non-zero constant.

We summarize the properties of the K\"{a}hler metric $\ol{\omega}$ which first appeared in~\cite{Fut3}.
\begin{prop}\label{prop:approx}
 Let $M$ be a Fano orbifold, and let $\mathbf{L}$ be a line bundle on $M$ with $\mathbf{L}^\times$ non-singular
 with $\mathbf{L}^p  =\mathbf{K}_M$ where $p=\alpha -1$.
 Suppose there is an $\eta$-Einstein Sasaki structure on $S\subset\mathbf{L}^\times$ compatible with the
 holomorphic structure, with K\"{a}hler potential $\frac{r^2}{2}$, $\rho^T =\kappa\omega^T$, $\kappa =2p$, and which
 is compatible with the natural $m+1$-form $\Omega$ on $\mathbf{L}^\times$ in that
 $\mathcal{L}_\xi \Omega =ip\Omega$.
 Then the metric $\ol{\omega}=\omega_\phi=i\partial\ol{\partial}G$ with $\phi$ as in (\ref{eq:prof3})
 and $G$ is defined by (\ref{eq:potent}) defines a Ricci-flat metric $g$ on $\mathbf{L}^\times$.
 Furthermore, $\ol{\omega}$ is complete at infinity and has Euclidean volume growth.
 The curvature tensor $R_g$ of $g$ satisfies $\|\nabla^k R_g\|_g =O(\rho^{-2-k})$, where $\rho$ denotes the distance from a fixed point.
 And $\ol{\omega}^{m+1}=\ol{c}\Omega\wedge\ol{\Omega}$, so it has a pole of order $2\alpha$ along $\infty$.
\end{prop}

That $\ol{\omega}$ has Euclidean volume growth follows easily for (\ref{eq:vol-part}).
The the asymptotic decay of $R_g$ follows from (\ref{eq:curvature}) and the ansatz (\ref{eq:cal-ansatz}) with
$\phi$ given by (\ref{eq:prof4}) and will be made clear in the proof of Proposition~\ref{prop:bound-geo}.

\section{proof of main theorem}

\subsection{normal bundle of a divisor}
Let $D\subset X$ be a divisor with $\alpha[D]=-K_X, \alpha >1.$  Let $N$ be the total space of the
normal bundle $N_D \cong [D]|_D$ to $D$, with $D\subset N$ the zero section.  Let
$\mathfrak{p}\subset\mathcal{O}(X)$ and $\tilde{\mathfrak{p}}\subset\mathcal{O}(N)$ be the ideal sheaves
of $D\subset X$ and $D\subset N$ respectively.  Denote by $D_{(\nu)}=(D,\mathcal{O}_\nu)$, where
$\mathcal{O}_\nu =\mathcal{O}(X)/{\mathfrak{p}^\nu}|_D$, the $\nu$-th infinitesimal neighborhood of $D$ in $X$.
Let $\tilde{D}_\nu$ the $\nu$-th infinitesimal neighborhood of $D$ in $N$.  We have
$D_{(2)}\cong \tilde{D}_{(2)}$.  If $\phi_k :D_{(k)}\cong \tilde{D}_{(k)}$ for $k\geq 2$,
is an isomorphism, then the obstruction to lifting to an isomorphism
$\phi_{k+1} :D_{(k+1)}\cong \tilde{D}_{(k+1)}$ is in $H^1(D,\Theta_X\otimes\mathcal{O}(-kD)|_D)$
(cf.~\cite{Gra} or~\cite{Gri}).  Thus we have condition (\ref{cond}) as necessary in order
to approximate a holomorphic neighborhood of $D\subset X$ with a neighborhood in the
normal bundle $N_D$ by a map whose jet is holomorphic to high order along $D$.

Since we assume condition (\ref{cond}) holds, we have an isomorphism
$\phi_\nu :D_{(\nu)}\cong \tilde{D}_{(\nu)}$ for arbitrary large $\nu\geq 2$.  Then $\phi_\nu$ defines a jet
\[J_D^\nu \phi_\nu \in J_D^\nu\diff_D(V,U),\]
along $D$, where $\diff_D(V,U)$ denotes diffeomorphisms fixing $D$ where $V$ and $U$ are small tubular
neighborhoods of $D$ in $X$ and $N$.  Provided $V$ and $U$ are sufficiently small, there is a diffeomorphism
$\psi\in\diff_D(U,V)$ with $J^\nu \psi=J^\nu \phi_\nu$ (cf.~\cite{Cer}, Ch. II).

We collect some vanishing results which can be used in applying Theorem~\ref{thm:main}.
Let $X$, $\dim X\geq 3$, be a Fano manifold and $D\subset X$ be a smooth divisor with $\alpha[D]=c_1(X) >0$
with $\alpha >1$.  Then $c_1(D)=(\alpha -1)[D]|_D>0$, so $D$ is Fano as well.
Suppose that $D$ is toric.
We have the following:
\begin{prop}\label{prop:cond}
Suppose either $\alpha\leq 2$, or $X$ is toric and $\dim X\geq 4$.  Then
$H^1(D,\Theta_X \otimes\mathcal{O}(-kD)|_D) =0$ for all $k\geq 2$.
\end{prop}
\begin{proof}
Suppose $\alpha\leq 2$.  We have the exact sequence on $D$,
\[0\rightarrow\Theta_D \rightarrow\Theta_X \rightarrow\mathcal{N}_{X/D}\rightarrow 0.\]
Using $\mathcal{N}_{X/D} =\mathcal{O}([D]|_D)$ we have
\[ \cdots\rightarrow H^1(D,\Theta_D(-kD))\rightarrow H^1(D,\Theta_X(-kD)|D)\rightarrow H^1(D,\mathcal{O}((1-k)D))\rightarrow\cdots. \]
By Kodaira-Serre duality, $H^1(D,\Theta_D(-kD))\cong H^{n-2}(D,\Omega^1((k+1-\alpha)D))$ which is zero
for $k+1-\alpha>0$ by the Bott vanishing theorem (cf. p.130 of~\cite{Oda}).
We have $H^1(D,\mathcal{O}((1-k)D))=0$ by Kodiara vanishing and the negativity of $[(1-k)D]$, for
$k\geq 2$.

The proof for the case with $X$ toric is a similar application of the Bott vanishing theorem.  Note that we are not
assuming that $D\subset X$ is an invariant embedding.
\end{proof}

Note that one can make use of some theorems on the existence of smooth divisors (cf.~\cite{PS}).
Suppose $X$ is a Fano manifold with $\ind X=r$, so $\mathbf{K}^{-1}_X =r\mathbf{H}$.  Then if either $n=3$ or $r\geq n-1$,
the linear system $|H|$ contains a smooth irreducible divisor.

\subsection{The Approximating metric}

We will construct an approximate metric, which will be used in the proof of Theorem~\ref{thm:main},
in each K\"{a}hler class in $H^2_c(Y,\R)$.

We have a diffeomorphism $\psi$ of $V\subset X$ with a neighborhood of infinity $U$ of
$\mathbf{L}= N_D^{-1} =[D]|_{D}^{-1}$, where $\mathbf{L}^p=\mathbf{K}_D$, $p=\alpha-1$,
whose $\nu$-jet is holomorphic along $D$ for any large $\nu$.
Let $G$ be a K\"{a}hler potential away from the zero section given in (\ref{eq:potent}) of the Ricci-flat
metric from section~\ref{subsect:calabi}.  Define $g=\psi^*G$.
Then $\omega =i\partial\ol{\partial}g$ is a K\"{a}hler form in a neighborhood of $D$ on $Y=X\setminus D$.
By shrinking $V$ we may assume $\omega$ is positive definite on $V$.
Let $V_r$ be the subset of $V$ defined by $V_r =\{x\in V : g(x)>r \}$.  Let $0<a<b$ be such that
$V_a \Subset V$.  Define a smooth function $\lambda :\R\rightarrow\R$ so that $\lambda(x)=x$ for $x\geq b$,
$\lambda(x)=\frac{b-a}{2}$ for $x\leq a$, and in the interval $(a,b)$ $\lambda' >0$ and $\lambda''>0$.
Then $h=\lambda\circ g$ extends to a smooth function on $Y=X\setminus D$.  Simple calculation shows that
$i\partial\ol{\partial}h \geq 0$ on $Y$, and $i\partial\ol{\partial}h>0$ on $V_a$.

Thus $Y$ is a 1-convex space, meaning that $Y$ carries an exhaustion function $h:Y\rightarrow [0,\infty)$
which is strictly plurisubharmonic, i.e. $i\partial\ol{\partial}h>0$, outside a compact set.
A 1-convex space is holomorphically convex.  We consequently have the \emph{Remmert reduction}
$\pi:Y\rightarrow W$ where $W$ is a Stein space (cf.~\cite{Gra}), and $\pi$ has the following
properties.
\begin{thmlist}
\item  $\pi$ is proper, surjective, and with connected fibers.
\item  $\pi_* \mathcal{O}_Y =\mathcal{O}_W$.
\item  The map $\pi^* :\mathcal{O}_W(W)\rightarrow\mathcal{O}_Y(Y)$ is an isomorphism.
\item  The exceptional set
\begin{equation}
A=\{y\in Y : \dim_{\C} \pi^{-1}(\pi(y))>0\}
\end{equation}
is the maximal compact analytic set of $Y$.
\end{thmlist}

We will need the following vanishing results.
\begin{prop}\label{prop:vanish}
Suppose $Y$ is a 1-convex manifold with $\mathbf{K}_Y^p$ trivial for some $p\in\N$.  Then $H^j(Y,\mathcal{O}_Y)=0$
for $j\geq 1$.  If $B=\{h<c\}$ is strongly pseudoconvex, then $H^j(Y\setminus\ol{B},\mathcal{O}_Y)=0$ for $1\leq j\leq n-2$.
\end{prop}
\begin{proof}
Suppose first that $\mathbf{K}_Y$ is trivial.  By the Grauert-Riemenschneider vanishing theorem~\cite{GraRie}
\begin{equation}
H^j(Y,\mathcal{O}_Y)= H^j(Y,\mathcal{O}(\mathbf{K}_Y))=0,\quad\text{for }j\geq 1.
\end{equation}

We have a nowhere vanishing section $\sigma\in\Gamma(\mathbf{K}_Y)$.  Since the singularity set of $W$ has codimension $n\geq 2$,
the dualizing sheaf satisfies $\omega_W =i_*\mathcal{O}(\mathbf{K}_Y)$, where $i: U\rightarrow W$ is the inclusion of the non-singular
set, and $\sigma\in\Gamma(\omega_W)$.  The singularity set of $W$ is discrete, and any singular point $p\in W$ has a
relatively compact neighborhood $V$ with $V\setminus\{p\}$ smooth.  Recall that a point $p\in W$ is a \emph{rational singularity}
if $(R^i \pi_* \mathcal{O}_Y)_x =0$ for $i>0$ where $\pi:Y\rightarrow W$ is any resolution of singularities.
It is easy to see that
\begin{equation}\label{eq:rat}
\int_V \sigma\wedge\ol{\sigma}<\infty.
\end{equation}
It is a result of~\cite{Bur} that an isolated singularity for which there is a non-vanishing holomorphic n-form on a deleted
neighborhood is rational if and only if (\ref{eq:rat}) holds.  Thus $W$ contains only rational singularities.

If $c>0$ is chosen large enough that $B=\{h<c\}$ is strongly pseudoconvex and $h$ is strictly plurisubharmonic on $Y\setminus B$,
then the exceptional set $A\subset B$.  Thus $\pi(B)$ is a strongly pseudoconvex neighborhood in $W$, which we also denote by $B$.
It is well known that rational singularities are Cohen-Macauley~\cite{AndAle}.
The vanishing theorem in~\cite[I.\S 3.\ Theorem 3.1]{BanSta}
implies that $H_{\ol{B}}^k(W,\mathcal{O}_W)=0$ for $k\leq n-1$, where $H^k_{\ol{B}}$ denotes cohomology with supports in $\ol{B}$.
Recall the exact sequence with cohomology with supports
\[\rightarrow H^{k-1}(W,\mathcal{O})\rightarrow H^{k-1}(W\setminus\ol{B},\mathcal{O})\rightarrow H^k_{\ol{B}}(W,\mathcal{O})
\rightarrow H^k(W,\mathcal{O})\rightarrow.\]
Since $W$ is Stein, we have
\[ H^j(Y\setminus\ol{B},\mathcal{O})=H^j(W\setminus\ol{B},\mathcal{O})=H^{j+1}_{\ol{B}}(W,\mathcal{O})=0,\ \text{for }1\leq j\leq n-2.\]

Now suppose that merely $\mathbf{K}_Y^q$ is trivial for some $q\in\N$.  Then there is a finite cover $\varpi :\tilde{Y}\rightarrow Y$ with
$\mathbf{K}_{\tilde{Y}}$ trivial and $\varpi^* h$ is strictly plurisubharmonic outside a compact subset of $\tilde{Y}$.
Clearly, the above proof works on $\tilde{Y}$.  Suppose $\beta\in\Omega^{0,j}(Y\setminus\ol{B})$ satisfies $\ol{\partial}\beta=0$
for $1\leq j\leq n-2$.  Then $\varpi^*\beta =\ol{\partial}\gamma$ for some $\gamma\in\Omega^{0,j-1}(\tilde{Y}\setminus\ol{B})$.
Let $G$ be the group of deck transformations of $\varphi$, and define $\ol{\gamma}=\frac{1}{|G|}\sum_{g\in G} g^*\gamma$.
Then it is easy to check that $\ol{\gamma}=\varpi^* \theta$ for some $\theta\in\Omega^{0,j-1}(Y\setminus\ol{B})$, and
$\ol{\partial}\theta =\beta$.  Thus $H^j(Y\setminus\ol{B},\mathcal{O})=0$ for $1\leq j\leq n-2$ in this case as well, and
the exact argument shows $H^j(Y,\mathcal{O})=0$ for $j\geq 1$ as well.
\end{proof}

\begin{lem}\label{lem:ddbar}
The $\partial\ol{\partial}$-lemma holds on $Y$, and for $n>2$ on $Y\setminus\ol{B}$ where $B$ is as in Proposition~\ref{prop:vanish}.
That is if $\beta$ is an exact real $(1,1)$-form on $Y$, then $\beta =i\partial\ol{\partial}f$ for $f\in C^\infty (Y)$.
And the analogous result holds on $Y\setminus\ol{B}$ for $n>2$.
\end{lem}
\begin{proof}
There exists an $\alpha\in\Omega^1$ with $d\alpha =\beta$.  Splitting into types we have $\alpha=\alpha^{1,0}+\alpha^{0,1}$
with $\alpha^{0,1}=\ol{\alpha}^{1,0}$, and $\ol{\partial}\alpha^{0,1}=0$.  Since $H^1(Y,\mathcal{O})=H^1(Y\setminus\ol{B},\mathcal{O})=0$,
there exists a $\gamma\in C^\infty$ with $\ol{\partial}\gamma=\alpha^{0,1}$.  Then one easily checks that
$\beta=i\partial\ol{\partial}2\im\gamma$.
\end{proof}

Suppose $\omega$ is a K\"{a}hler form on $Y$ with $[\omega]\in H^2_c(Y,\R)$.
Since $B_b :=\{x\in Y: h(x)<b\}$ is strongly pseudoconvex, we have $\omega|_{Y\setminus\ol{B}} =i\partial\ol{\partial}f$ with
$f\in C^\infty(Y\setminus\ol{B})$.
We define a cut-off function $\eta:Y\rightarrow[0,1]$ as follows.  Choose $c,d$ so that $b<c<d$ and define $\eta(x)=1$ for $x\leq c$,
$\eta(x)=0$ for $x\geq d$, and define $\eta$ to be decreasing with values in $(0,1)$ on $(c,d)$.
We define
\begin{equation}
\omega_0 =
\begin{cases}
\omega +Ci\partial\ol{\partial}h,  & \text{ on } B \\
i\partial\ol{\partial}\bigl((\eta\circ h) f\bigr) +Ci\partial\ol{\partial}h, & \text{ on } Y\setminus B. \\
\end{cases}
\end{equation}
For $C>0$ sufficiently large $\omega_0$ is a K\"{a}hler form, and clearly $[\omega_0]=[\omega]$.

\subsection{Monge-Amp\`{e}re equation}

Let $\sigma\in\Gamma(\mathbf{K}_X)$ be a section with a pole along $D$, of order $\alpha$.
Thus $\sigma\wedge\ol{\sigma}$ has a pole of order $2\alpha$.  Also, $\omega_0 ^n$ has a pole of order
$2\alpha$ along $D$. Thus
\begin{equation}
f=\log\left(\frac{\sigma\wedge\ol{\sigma}}{\omega_0^n}\right)
\end{equation}
extends to a smooth function on $X$.  We have $i\partial\ol{\partial}f=\ric(\omega_0)$ which is zero along $D$.
Thus $f$ is constant on $D$, and we may assume $f$ vanishes on $D$.
Furthermore, $\partial f|D\in H^0(D,\mathcal{O}(N^*))$.  And since $N^*$ is negative, $\partial f$ vanishes
along $D$.  Using the negativity of of $N^{-k},k\geq 1$, and that $i\partial\ol{\partial}f=\ric(\omega_0)$ vanishes
to order $\nu-4$ on $D$, one can show that the derivatives of $f$ up to order $\nu-2$ vanish along $D$.

Then we have the following properties of the approximating metric $\omega_0$.
\begin{prop}\label{prop:approx-metric}
 The form $\omega_0$ defines a complete K\"{a}hler metric $g_0$ on $Y$ such that
 \[ \ric(\omega_0) =i\partial\ol{\partial}f, \]
where $f$ is a smooth function on $X$ vanishing along $D$, and whose derivatives up to order $\nu-2$ vanish along $D$.
Furthermore, the curvature tensor satisfies $\|\nabla^k R(g_0)\|_{g_0} =O(\rho^{-2-k})$, where $\rho$ is the distance from
a fixed point.
\end{prop}

The following theorem is due to G. Tian and S.-T. Yau.  The final statement on the curvature decay follows from
~\cite{BKN}.
\begin{prop}[\cite{TY2}]\label{prop:TY}
Let $\omega_0$ be the K\"{a}hler metric on $Y=X\setminus D$ constructed above.  And let $f$ be as above
with $\ric(\omega_0)=i\partial\ol{\partial}f$.  Then the Monge-Amp\`{e}re equation
\begin{equation}\label{eq:monge}
\left(\omega_0 +i\partial\ol{\partial}\phi\right)^n =e^f\omega_0^n ,
\end{equation}
has a smooth solution $\phi\in C^\infty(Y)$ where $\phi$ converges uniformly to zero at infinity,
is bounded in $C^{2,\frac{1}{2}}$, and thus $\omega=\omega_0 +i\partial\ol{\partial}\phi$ satisfies
$c^{-1}\omega_0\leq\omega\leq c\omega_0$, for some $c>0$.

Thus $\omega$ is the K\"{a}hler form of a complete Ricci-flat K\"{a}hler metric $g$ on $Y$.
Furthermore, $g$ has Euclidean volume growth, and $\|R_g\|_g =O(\rho^{-2})$ where
$\rho(x)=\dist(o,x)$.  If $\|R_g\|_g =O(\rho^{-k})$ for $k>2$, then $(Y,g)$ is K\"{a}hler ALE
of order $2n$.  In which case $\|R_g\|_g =O(\rho^{-2n-2})$.
\end{prop}

By ALE of order $m$ we mean the following.  There exists a compact subset $K\subset Y$, a finite group
$\Gamma\subset O(2n)$ acting freely on $\R^{2n}\setminus\{0\}$, and a ball $B_R(0)\subset\R^{2n}$ of radius $R>0$.
So that there is a diffeomorphism $\chi:(\R^{2n}\setminus B_R(0))/\Gamma \rightarrow Y\setminus K$ and
\begin{equation}
\| \nabla^k \chi^*g -\nabla^k h\|_{h} =O(r^{-m-k}),
\end{equation}
where $h$ is the flat metric and $\nabla$ its covariant derivative.

We say that $Y$ is K\"{a}hler ALE if in addition we have $\R^{2n}=\C^n$ with the standard
complex structure $J_0$ and $\Gamma\subset U(n)$.  And if $J$ is the complex structure on $Y$ we have
\begin{equation}\label{eq:ALE-Kahler}
\| \nabla^k \chi^*J -\nabla^k J_0 \|_{h} =O(r^{-m-k}),
\end{equation}
and Ricci-flatness implies that $\Gamma\subset SU(n)$.  It is known that if an ALE manifold $Y$ is K\"{a}hler one
can choose coordinates at infinity $\chi:(\C^n \setminus B_R(0))/\Gamma \rightarrow Y\setminus K$ so that
(\ref{eq:ALE-Kahler}) holds.

\subsection{asymptotic properties of the metric}
We want to improve on the asymptotic behavior of $\phi$ in Proposition~\ref{prop:TY}.
First we need coordinates on $Y$ for which the metric has good bounds.   Then we will apply Schauder estimates
to the solution to Proposition~\ref{prop:TY}.  In the following $\rho$ will denote the distance from a fixed point
$o \in Y$ outside a compact set $K\subset Y$ containing $o\in Y$ and extended to a continuous function on all of $Y$.
\begin{defn}\label{defn:bound-geo}
A holomorphic coordinate chart $(U,z_1,\ldots, z_n)$ centered at $x\in Y$ in a K\"{a}hler manifold
is \emph{bounded of order $\ell$} with bound $(R,c, c_1,\ldots,c_{\ell})$ if
\begin{thmlist}
 \item  using the Euclidean coordinate distance the ball $B_R(x)\subseteq U$,
 \item  $\frac{1}{c}\delta_{ij}\leq g_{i\ol{j}}\leq c\delta_{ij}$ if $g_{i\ol{j}}$ denotes the metric tensor in $(U,z_1,\ldots, z_n)$,
 \item  and for any multi-indices $\alpha,\beta$ with $|\alpha|+|\beta|\leq\ell$
\[ \left| \frac{\partial^{|\alpha|+|\beta|} g_{i\ol{j}}}{\partial z^{\alpha}\partial\ol{z}^{\ol{\beta}}}\right|\leq c_{|\alpha|+|\beta|}.\]
\end{thmlist}

We say that $(Y,g)$ has \emph{bounded geometry of order $\ell$} if there are constants $(R,c, c_1,\ldots,c_{\ell})$
so that each $x_0 \in Y$ has a coordinate neighborhood $(U,z_1,\ldots, z_n)$ bounded of order $\ell$ with bound
$(R,c, c_1,\ldots,c_{\ell})$ for $\rho(x_0)^{-2} g$.
\end{defn}

\begin{prop}\label{prop:bound-geo}
 Let $(Y,\omega_0 )$ be the K\"{a}hler metric in Proposition~\ref{prop:approx-metric}.  Then for any
 positive integer $\ell$, $(Y,\omega_0 )$ has bounded geometry of order $\ell$ if $\nu$ in
 Proposition~\ref{prop:approx-metric} is chosen sufficiently large.
\end{prop}
\begin{proof}
We first show that the metric $\ol{\omega}=\omega_\phi$ of Proposition~\ref{prop:approx} has bounded
geometry of any order $\ell$.  Let $r_0 >0$ be sufficiently large so that the set
$Y_{r_0} :=\{x\in\mathbf{L}: r(x)\geq r_0\}\subset U$.  We denote
the metric on $U$ of (\ref{eq:cal-ansatz}) by $\omega_{\phi,a}$ with $\phi$ and $\tau$ as in (\ref{eq:prof3}) and
(\ref{eq:prof4}) with constant $a\geq 0$.  Since $S_{r_0}=\{x\in\mathbf{L}: r(x)= r_0\}$
is compact it can be covered by finitely many charts satisfying Definition~\ref{defn:bound-geo},
with the bounds $\psi_x ^*g$ uniform for $a$ in $[0,a_0]$.
By shrinking $R$ we get a chart $\psi_x :U_x \rightarrow U$ satisfying
Definition~\ref{defn:bound-geo} for every $x\in S_{r_0}$.

Now for $b>1$ let $R_{b}: U\rightarrow U$ be the map generated by the action of
$r\frac{\partial}{\partial r} -i\xi$.  Thus $R_{b}(Y_{r_0})=Y_{br_0}$.    Simple calculation shows that
$R_b^* \omega_{\phi,a} = e^{\frac{\kappa b}{m+1}} \omega_{\phi,\frac{a}{b}}$.
The distance $\rho$ is equivalent to $e^{\frac{\kappa t}{2(m+1)}}$, i.e.
$c_1 \rho \leq e^{\frac{\kappa t}{2(m+1)}} \leq c_2 \rho$ for $c_1, c_2 >0$.
Then $U_x \overset{\psi_x}{\rightarrow} U \overset{R_b}{\rightarrow} U$ defines the required coordinate
chart centered at $bx$.

Recall we are considering $\mathbf{L}^\times$ with a Sasaki-Einstein structure on the link $S\subset\mathbf{L}^\times$ which we denote
$(g_1,\xi_1,\eta_1,\Phi_1)$, with K\"{a}hler potential $\frac{r_1^2}{2}$.  Denote by
$(g_0,\xi_0,\eta_0,\Phi_0)$ the standard Sasaki structure on $\mathbf{L}^\times$ satisfying Proposition~\ref{prop:CY-cond}
with K\"{a}hler potential $\frac{r_0^2}{2}$.
Let $T^k$ be the torus acting on $\mathbf{L}^\times$ whose Lie algebra $\mathfrak{t}$ contains both
$\xi_0$ and $\xi_1$.  Define $Z=\mathbb{P}(\mathbf{L}\oplus\C)$.  Let $\mathbf{H}$ be the tautological
orbifold bundle on $Z$ restricting to the hyperplane bundle on each fiber of $\pi:Z\rightarrow D$.
Then the bundle $\mathbf{E}=\mathbf{H}\otimes\pi^* \mathbf{K}_D^{-j}$ is positive for $j>>0$.
Clearly, $T^k$ and its complexification $(\C^*)^k$ acts on $\mathbf{E}$.  By the Baily embedding theorem~\cite{Ba}
we have an embedding $\iota_{\mathbf{E}^\ell} :Z\rightarrow\cps^N$ for $\ell>>0$.
It follows that $\psi :(\C^*)^k \times Z\rightarrow Z$ is an algebraic action.
In particular, $\psi$ extends to a rational map $\ol{\psi}:(\cps^1)^k \times Z\rightarrow Z$.
If $z_1\frac{\partial}{\partial z_1},\ldots, z_k\frac{\partial}{\partial z_k}$ is the basis of the Lie
algebra $\mathfrak{t}_{\C}$ of $(\C^*)^k$, then
$\frac{\partial}{\partial r_1} =\sum_j b_j z_j\frac{\partial}{\partial z_j}$ for $(b_1,\ldots,b_k)\in\R^k$
defines an embedding $\R_+ \rightarrow (\C^*)^k$.  And the restriction
$\tilde{\phi}:\R_+ \times S \rightarrow Z$ is a diffeomorphism onto $\mathbf{L}^\times$.  Thus
if $P_1 :\R_+ \times S \rightarrow\R_+$ denotes the projection, then $r_0 =P_1 \circ\tilde{\phi}^{-1}$.

Let $(W,z_1,\ldots,z_{m})$ be a coordinate neighborhood on $D$ trivializing $\mathbf{L}$ with fiber coordinate
$w$, where $m=n-1$ as above.  Then $(z_1,\ldots, z_{n-1}, z_{n})$ with $z_{n} =\frac{1}{w}$ defines a
coordinate $U$ chart intersecting $D\subset Y$, as the section at infinity, with $D\cap U=\{z\in V: z_{n}=0\}$.
And $D$ can be covered by finitely many such coordinate charts.  Since $\psi$ is a rational map, in $U$ we have
\begin{equation}\label{eq:potent-expr}
r_1 =\frac{1}{|z_{n}|^q}f(z),
\end{equation}
with $q\in\R_+$ and $f(z)$ a smooth function.

As before denote $t_0 =\log r_0$ and $t_1 =\log r_1$.  Then from (\ref{eq:cont}) we have
\begin{equation}\label{eq:reeb-bound}
\frac{\partial t_1}{\partial t_0}=\frac{r_0}{r_1}\frac{\partial r_1}{\partial r_0} =\frac{1}{2}\eta_1(\xi_0)>c>0,
\end{equation}
for a constant $c$.  Since $\left[\frac{\partial}{\partial r_1},\frac{\partial}{\partial r_0}\right]=0$,
it follows from integrating (\ref{eq:reeb-bound}) and a similar bound for $\frac{\partial t_0}{\partial t_1}$
that
\begin{equation}\label{eq:potent-bound}
C_1 r_0^{q_1} \leq r_1 \leq C_2 r_0^{q_2},
\end{equation}
for $q_1 ,q_2 \in\R_+, q_1 >q_2$.
We denote by $\rho_0$ and $\rho_1$ the distance from a fixed point $x\in S\subset\mathbf{L}^\times$
with respect to the metrics $g_0$ and $g_1$.  Then we have
\begin{equation}\label{eq:dist-bound}
C_1 \rho_0 (y)^{q_1} \leq \rho_1 (y) \leq C_2 \rho_0 (y)^{q_2},\quad\text{for}\quad y\in\mathbf{L}^\times.
\end{equation}
And we have a bound as in (\ref{eq:dist-bound}) for the distance $\rho$ of $\ol{\omega}$.

We have the diffeomorphism $\psi:V\rightarrow U$, where $V\subset X$ is a neighborhood of $D\subset X$,
whose jet is holomorphic to arbitrarily high order along $D$.
Let $\ol{g}$ denote the metric on $U\subset\mathbf{L}^\times$ with K\"{a}hler form
$\ol{\omega}$.  For notational simplicity denote the metric $\psi^* \ol{g}$ on $V\subset X$
also by $\ol{g}$.  We also have the metric on $V$ with K\"{a}hler form $dd^c \psi^* G$ which
we denote $g$.  It follows from (\ref{eq:potent-expr}) and (\ref{eq:dist-bound}) that for $k_0 >0$
if $\psi:V\rightarrow U$ is chosen to have holomorphic jet along $D$ of sufficiently high
order $\mu$, then
\begin{equation}\label{eq:approx-bound}
\|\nabla^j (g-\ol{g})\| =O(\rho^{-\alpha-2-j}),\quad\text{for}\quad 0\leq j\leq k_0,
\end{equation}
where $\alpha >2n$ and $\nabla$, the norm, and the distance $\rho$ is with respect to $\ol{g}$.

Let $(U'_x ,z_1,\ldots,z_{n}) ,x\in U'$ be the family of neighborhoods on $U'\subset\mathbf{L}^\times$ satisfying
Definition~\ref{defn:bound-geo}.  Then the $w_j =\phi^* z_j$ form non-holomorphic coordinates on
$U_v =\phi^{-1}(U'_x), \phi(v)=x$ satisfying the conditions in~\ref{defn:bound-geo}
for the metric $\rho(v)^{-2}\phi^* \ol{g}$, and thus also for $\rho(v)^{-2}g$ by (\ref{eq:approx-bound}).
In the following we will consider each $U_x$, $x\in V$, with the metric $\rho(x)^{-2}g$, and
$\rho$ is the distance with respect to this metric.

Consider the real coordinates $w_j =x_j +iy_j ,j=1,\ldots,n,$ on $U_v ,v\in V$, and denote $y_j$ by
$x_{n+j}$.  First make a linear change of coordinates so that each $dw_j |_v \in\Lambda^{1,0}$,
i.e. $\ol{\partial}w_j =0$ at $v$.
Then make the change of coordinates $x_k \mapsto x_k +\frac{1}{2}\Gamma^k_{ij}x_i x_j$, so that
$\Gamma^k_{ij}(v)=0$.  Thus $\nabla dw_j =0$ at $v$.

Let $\rho$ be the distance from $x\in U_x$.  Since  $\underset{V}{\sup}\|R(g)\|_g \leq C$,
by the Hessian comparison theorem (cf.~\cite{SiuYau}) there are positive constants $r$ and $C_1$
so that restricting to $B_r=\{y :\rho(y)<r\}\subset U_x$
the functions $\rho^2$ and $\log(\rho^2)$ are plurisubharmonic and
\begin{equation}
i\partial\ol{\partial}\rho^2 \geq C_1 >0,\quad\text{on}\quad B_r.
\end{equation}
Here $r,C_1$ can be chosen uniformly for every chart $U_x$.  Replace each $U_x, x\in V$ with its
restriction to $B_r$.  Thus, in particular, each $U_x$ is strictly pseudoconvex.

Let $P_{0,1}$ be the projection of $\Lambda^1$ onto its $(0,1)$ component, then $\ol{\partial} w_j =P_{0,1} dw_j$.
Since $\nabla$ preserves types, $\nabla\ol{\partial} w_j =0$ at $x$.  Let $t$ denote the radial distance
from $x\in U_x$ along a geodesic.  Then $\frac{d^i}{dt^i}g(\ol{\partial} w_j,\ol{\partial} w_j)=0$ at
$x\in U_x$ for $0\leq i\leq 4$.  Expanding gives
\begin{equation}\label{eq:dbar-bound1}
\begin{split}
\frac{d^4}{dt^4}g(\ol{\partial} w_j,\ol{\partial} w_j) & =
g(P_{0,1}\nabla^4_t dw_j, P_{0,1}dw_j)+g(P_{0,1}\nabla^3_t dw_j, P_{0,1}\nabla_t dw_j)+
g(P_{0,1}\nabla^2_t dw_j, P_{0,1}\nabla^2_t dw_j)\\
& \quad +g(P_{0,1}\nabla_t dw_j, P_{0,1}\nabla^3_t dw_j)+g(P_{0,1} dw_j, P_{0,1}\nabla^4_t dw_j)\\
& \leq g(\nabla^4_t dw_j, dw_j)+g(\nabla^3_t dw_j, \nabla_t dw_j)+
g(\nabla^2_t dw_j, \nabla^2_t dw_j)\\
& \quad +g(\nabla_t dw_j, \nabla^3_t dw_j)+g(dw_j, \nabla^4_t dw_j)\\
& \leq C, \\
\end{split}
\end{equation}
where $C$ depends on the uniform bounds of the metric and its derivatives in the coordinates
$(w_1,\ldots,w_{n})$. Integrating (\ref{eq:dbar-bound1}) four times gives
\begin{equation}\label{eq:dbar-bound2}
g(\ol{\partial} w_j,\ol{\partial} w_j)\leq C\rho^4.
\end{equation}
By H\"{o}rmander's $L^2$-estimate with weight function $\phi=(2n+3)\log\rho +2\log (1+\rho^2)$,
which is strictly plurisubharmonic after possibly shrinking $R$, there is a function
$u_j\in C^\infty (U_x)$ so that $\ol{\partial}u_j=\ol{\partial}w_j$ and
\begin{equation}
\int_{U_x} |u_j|^2 e^{-\phi} d\mu \leq C_1\int_{U_x}|\ol{\partial}w_j |^2 e^{-\phi}d\mu \leq C_2,
\end{equation}
where $C_2 >0$ follows from (\ref{eq:dbar-bound2}) and $C_1 ,C_2 >0$ are independent of the neighborhood $U_x$.
Because of the singularity of $\phi$, $u_j =du_j =0$ at $x$.  Then the functions
$z_j =w_j -u_j ,j=1,\ldots, n,$ give a holomorphic coordinate system in a neighborhood of $x\in U_x$ which,
after shrinking $R$ if necessary, we denote by $U_x$ again.

In each $U_x$ we consider the Laplacian $\Delta$ with respect to the K\"{a}hler metric $\rho(x)^{-2} g$.
Let $U'_x \subset U_x$ be the ball of radius $R'< R$ with respect to this metric.
Then for $\Delta u_j =\Delta w_j =f$ we have the Sobolev estimate
\begin{equation}\label{eq:ellip-sob}
\|u_j \|_{L^2_{s+2}(U'_x)} \leq C\left(\|f\|_{L^2_{s}(U_x)} +\|u\|_{L^2(U_x)}\right),
\end{equation}
where $C$ is independent of $U_x$.  Also the Sobolev inequality gives the H\"{o}lder bound
\begin{equation}\label{eq:sob-ineq}
\| u_j\|_{C^{s+1 -n,\gamma}(U'_x)} \leq C' \|u_j \|_{L^2_{s+2}(U'_x)},
\end{equation}
for $0<\gamma <1$, where $C'$ depends only on $s$ and $\gamma$.  By choosing $\phi:V\rightarrow U$ to have
jet along $D$ which is holomorphic of sufficiently high order $\mu$ and shrinking $R'>0$ if necessary
we may bound $u_j$ in $C^{\ell,\gamma}$ by an arbitrarily small constant in each $U'_x$.
So in each $U'_x$ the $z_j =w_j -u_j ,j=1,\ldots, n,$ define a system of coordinates
satisfying the conditions of Definition~\ref{defn:bound-geo}.
\end{proof}

In the following we will denote the metric of Proposition~\ref{prop:approx-metric} by $g_0$ with K\"{a}hler form $\omega_0$.
We will prove that the Ricci-flat metric $g$ with K\"{a}hler form $\omega$ of Proposition~\ref{prop:TY} has nice asymptotic
properties.
\begin{prop}\label{prop:asymp}
If $\phi$ is the solution to (\ref{eq:monge}) of Proposition~\ref{prop:TY} for any $\delta >0$ there is a compact $K\subset Y$ with	
\[ -C_1 \rho(y)^{-2n +2+\delta}\leq \phi\leq C_2 \rho(y)^{-2n+2+\delta},\quad\text{for}\quad y\in Y\setminus K,   \]						
where $C_1,C_2 >0$.
\end{prop}
\begin{proof}
Outside a compact set the metric $g_0$ is close according to the estimate (\ref{eq:approx-bound}) to the Calabi ansatz $\omega_\phi$
with profile function (\ref{eq:prof4}).  The distance from a fixed point $\rho(y)$ is equivalent to $r' =r^{\frac{\kappa}{2n}}$.
Then a complicated but straightforward calculation shows that with $\tau =C_1 r^\alpha$ we have
\begin{equation}
 (\omega_\phi +i\partial\ol{\partial}\tau)^n < \omega_\phi^n,
\end{equation}
for $\alpha > -\frac{(n-1)\kappa}{n}$.  And similarly, if $\tau =-C_2 r^\alpha$
\begin{equation}
 (\omega_\phi +i\partial\ol{\partial}\tau)^n >\omega_\phi^n,
\end{equation}
for $\alpha > -\frac{(n-1)\kappa}{n}$ (cf.~\cite[Lemma 5.2]{vC2}).  Then for $C_1, C_2 >0$ chosen sufficiently large we can apply
the maximum principle.  But first observe that $f$ in Proposition~\ref{prop:approx-metric} vanishes to arbitrarily high order along $D$.
Then from (\ref{eq:dist-bound}) and (\ref{eq:approx-bound}) if we denote $\psi^*\tau$ by $\tau$ also, we have in a neighborhood $V$ of $D$
\begin{equation}\label{eq:max-pri}
 (\omega_0 +i\partial\ol{\partial}\tau)^n \leq \omega_0^n e^f.
\end{equation}
And for $\tau =-C_2 \psi^*(r^\alpha)$ we have on $V$
\begin{equation}
 (\omega_0 +i\partial\ol{\partial}\tau)^n \geq \omega_0^n e^f.
\end{equation}
Thus if $C_1, C_2 >0$ are chosen sufficiently large, since $\phi$ converges uniformly to zero at infinity, we may apply
the maximum principle as in~\cite{CheY}) to get the inequalities.

Alternatively, one can apply the non-linear version of the maximum principle as follows.
We have (\ref{eq:max-pri}) in $V$.  Choose $C_1 >0$ large enough that $\tau =C_1 r^\alpha \geq\phi$ on
$\partial V$.  Since $\lim_{\rho\rightarrow\infty}\phi- \tau=0$ and $\phi-\tau \leq 0$ on $\partial V$, either
$\phi-\tau\leq 0$ on $V$ or $\phi-\tau$ attains its maximum at $x\in V$.  If the latter, then~\cite[Theorem 17.1]{GiTr} applied
in a coordinate neighborhood $(U,z_1,\ldots,z_n)$ centered at $x\in V$ implies that $\phi-\tau$ is constant on $U$.
And by taking other coordinate neighborhood one sees that $\phi-\tau$ is constant, and therefore $\phi=\tau$.
The lower bound is obtained similarly.
\end{proof}

For the following proposition we will need \emph{weighted} H\"{o}lder spaces.
For $\beta\in\R$ and $k$ a nonnegative integer we define $C_\beta^k(Y)$ to be the space of continuous functions $f$
with $k$ continuous derivatives for which the norm
\[ \|f\|_{C_\beta^k} := \sum_{j=0}^k \underset{Y}{\sup}|\rho^{j-\beta}\nabla^j f|\]
is finite.  Then $C_\beta^k(Y)$ is a Banach space with this norm.

Let $\inj(x)$ be the injectivity radius at $x\in Y$, and $d(x,y)$ the distance between $x,y\in Y$.  Then for
$\alpha,\beta\in\R$ and $T$ a tensor field define
 \[ [T]_{\alpha,\beta} :=\sup_{\substack{x\neq y\\ d(x,y)<\inj(x)}}\left[\min(\rho(x),\rho(y))^{-\beta}\cdot\frac{|T(x)-T(y)|}{d(x,y)^\alpha}\right],  \]
where $|T(x)-T(y)|$ is defined by parallel translation along the unique geodesic between $x$ and $y$.

For $\alpha\in(0,1)$ define the \emph{weighted H\"{o}lder space} $C_\beta^{k,\alpha}(Y)$ to be the set of
$f\in C_\beta^k(Y)$ for which the norm
\begin{equation}\label{eq:weight-hol}
\|f\|_{C_\beta^{k,\alpha}} :=\|f\|_{C_\beta^k} +[\nabla^k f]_{\alpha,\beta-k-\alpha}
\end{equation}
is finite.  Then $C_\beta^{k,\alpha}(Y)$ is a Banach space with this norm.

In the case of $(Y,g_0)$ with bounded geometry of order $\ell >k$ it is not difficult to see that the norm
(\ref{eq:weight-hol}) is equivalent to taking the norm $\rho(x)\|\cdot\|_{C^{k,\alpha}}$ in each coordinate system
$(U_x,z_1,\ldots,z_n)$, where $\|\cdot\|_{C^{k,\alpha}}$ is the H\"{o}lder norm with respect to the usual Euclidean metric.

\begin{prop}\label{prop:decay}
For any $\delta >0$ the solution to (\ref{eq:monge}) of Proposition~\ref{prop:TY} satisfies
\[ \|\nabla ^j \phi \|_{g_0} =O(\rho ^{-2n +2 -j+\delta}),\quad\text{for}\quad j\leq k.\]
\end{prop}

\begin{proof}
Let $\{(U_x,\varphi_x , z_1, \ldots,z_n)\}_{x\in Y}$ be a system of holomorphic charts giving $(Y,g_0)$ bounded geometry of order
$\ell$ as in Definition~\ref{defn:bound-geo}, where $\varphi_x :B_R \rightarrow U_x \subset Y$ is the chart centered at $x\in Y$.
We define a chart centered at $x\in Y$ depending on $r>0$, $\psi_{x,r}:B_R \rightarrow U_x$.
Let $R_a :B_R \rightarrow B_R$ be scalar multiplication by $a$.  Then define $\psi_{x,r} := \varphi_x \circ R_{r\rho(x)^{-1}}$.
Note that for $r=1$ the charts $\psi_{x,r}$ are bounded of order $\ell$ for $g_0$, and for $r=\rho(x)$ the charts
$\psi_{x,r} =\varphi_x$ are bounded of order $\ell$ for $\rho^{-2}(x)g_0$.

We first must show that for each $x\in Y$ we have $\|\psi_{x,1}^*\phi \|_{C^{\ell,\alpha}} \leq C$ for some $C>0$ where
the H\"{o}lder norm is with respect to the Euclidean metric $h$ on $B_R \subset\C^n$.  Let $\Delta'$ denotes the Laplacian
with respect to the metric $g$ associated to $\omega=\omega_0 +i\partial\ol{\partial}\phi$; let $\Delta$, $\nabla$, and
$R$ denote the Laplacian, connection, and curvature of $g_0$.  Then the H\"{o}lder bound follows from a bootstrapping
argument involving the equation
\begin{equation}
\Delta'\Delta\phi =g'^{\alpha\ol{\mu}}g'^{\gamma\ol{\beta}}(\nabla^\nu \nabla_{\gamma\ol{\mu}}\phi)(\nabla_\nu \nabla_{\alpha\ol{\beta}}\phi) -\Delta f+g'^{\lambda\ol{\mu}} {R^{\alpha\ol{\beta}}}_{\lambda\ol{\mu}}\nabla_{\alpha\ol{\beta}}\phi- g'^{\nu\ol{\mu}}{R^{\ol{\lambda}}}_{\ol{\mu}}\nabla_{\ol{\lambda}\nu}\phi.
\end{equation}
The complete proof is given in~\cite[Theorem 3.21]{vC4}.  The arguments are also explained in~\cite{Joy}.

We may assume that $f$ is chosen to vanish to high enough order in Proposition~\ref{prop:approx-metric} that $f\in C^{k,\alpha}_\beta(Y)$
where $\beta<-2n$.   Suppose now that $\phi\in C^0_{\gamma}(Y)$ and $\|\phi\|_{C^0_\gamma}=P$ for some $\beta+2\leq\gamma<0$ and $P>0$.
In the following we will also assume that $\ell\geq k+2$.
The proof of the proposition will be based on the following lemma due to D. Joyce~\cite{Joy}.
\begin{lem}
Let $N_1,N_2>0$ and $\lambda\in[0,1]$.  Then there exists an $N_3>$
depending on $(X,g,J),\gamma, \alpha,\ \beta,\ k, P$ and $N_1,\ N_2, \ \lambda$, so that the following holds.

Suppose $\|f\|_{C^{k,\alpha}_\beta}\leq N_1$ and
\begin{equation}
\begin{split}\label{eq:Sch-prop-1}
& \|\nabla^j dd^c\phi\|_{C^0_{-j\lambda}} \leq N_2, \text{ for }j=0,\ldots, k,\\
& \text{ and } \bigl[\nabla^k dd^c\phi\bigr]_{\alpha,-(k+\alpha)\lambda} \leq N_2.
\end{split}
\end{equation}
Then the following inequalities hold, where the norms exist,
\begin{equation}
\begin{split}\label{eq:Sch-prop-2}
& \|\nabla^j \phi\|_{C^0_{\gamma-j\lambda}} \leq N_3, \text{ for }j=0,\ldots, k+2,\\
& \text{ and } \bigl[\nabla^{k+2} \phi\bigr]_{\alpha,\gamma-(k+2+\alpha)\lambda} \leq N_3.
\end{split}
\end{equation}
\end{lem}
\begin{proof}
We define an operator $P_{x,r} :C^{k+2,\alpha}(B_R)\rightarrow C^{k,\alpha}(B_R)$ by
\begin{equation}\label{eq:Sch-op}
(\psi_{x,r})_* \bigl[P_{x,r} u\bigr]\omega_0^n =-r^2 dd^c \bigl[(\psi_{x,r})_* u\bigr]\wedge(\omega_0^{n-1} +\cdots+\omega^{n-1}).
\end{equation}
Thus $P_{x,r}$ is an elliptic operator, and it follows from (\ref{eq:monge}) that
\begin{equation}\label{eq:Sch-eq}
P_{x,r}\bigl(\psi^*_{x,r}\phi \bigr)=r^2 \bigl(1-\psi_{x,r}^*(e^f)\bigr).
\end{equation}

Now set $r=\rho(x)^\lambda$ for $\lambda\in [0,1]$.
By shrinking $R>0$ if necessary we may suppose that
\begin{align}
&\frac{1}{2}\rho(x)\leq\rho(y)\leq 2\rho(x), \quad\text{for all }y\in\psi_{x,r}(B_R), \label{eq:bd-ine}\\
&\|r^{-2} \psi_{x,r}^*g_0 -h\|_{C^{k,\alpha}}\leq\frac{1}{2},\label{eq:bd-Eu}
\end{align}
for each $x\in Y$ where $h$ is the Euclidean metric on $B_R \subset\C^n$.  This makes use of the boundedness of
the coordinate charts $\varphi_x$.
In addition, it follows from (\ref{eq:bd-Eu}) and (\ref{eq:Sch-prop-1}) that
\begin{equation}\label{eq:bd-sym}
\|r^{-2} \psi_{x,r}^*\omega \|_{C^{k,\alpha}}\leq\C,
\end{equation}
for a constant $C>0$ depending on $k,\alpha$ and $N_2$.

The operator (\ref{eq:Sch-op}) written in real coordinates on $B_R$ is of the form
\[ P_{x,r} u =a^{ij} \frac{\partial^2 u}{\partial x_i \partial x_j}.\]
And it follows from (\ref{eq:bd-Eu}) and (\ref{eq:bd-sym}) that there are constants $\kappa$ and $K$
independent of $x\in Y$ so that $|a^{ij}\xi_i \xi_j|\geq\kappa |\xi|^2$ and $\|a^{ij}\|_{C^{k,\alpha}}\leq K$
for all $i,j =1,\ldots,2n$.

It follows from (\ref{eq:Sch-eq}) and the Schauder interior estimates~\cite[Th. 6.2,\ Th.6.7]{GiTr} that there is a constant
$C>0$, independent of $x\in Y$, so that
\begin{equation}
\|\psi_{x,r}^* \phi|_{B_{\frac{R}{2}}} \|_{C^{k+2,\alpha}} \leq C\left(r^2 \|1-\psi_{x,r}^*(e^f)\|_{C^{k,\alpha}} +\|\psi_{x,r}^* \phi\|_{C^0} \right)
\end{equation}
Let $U_2 =\psi_{x,r}(B_R)$ and $U_1 =\psi_{x,r}(B_{\frac{R}{2}})$.
Using the boundedness of the charts $\varphi_x$ we have, after possibly increasing $C>0$,
\begin{equation}\label{eq:Sch-bd}
\begin{split}
& \sum_{j=0}^{k+2} r^j\|\nabla^j \phi|_{U_1}\|_{C^0} +r^{k+2+\alpha}\left[\nabla^{k+2}\phi|_{U_1}\right]_\alpha \leq \\
 & C \left[\sum_{j=0}^k r^j\|\nabla^j f|_{U_2}\|_{C^0} +r^{k+2+\alpha}\bigl[\nabla^k f|_{U_2}\bigr]_\alpha +\|\phi|_{U_2}\|_{C^0} \right].
\end{split}
\end{equation}
From (\ref{eq:bd-ine}) and $r=\rho(x)^\lambda$ we get  $\|\phi|_{U_2}\|_{C^0} \leq P 2^{-\gamma}\rho(x)^\gamma$,
$\|\nabla^j f|_{U_2}\|_{C^0} \leq N_1 2^{j-\beta}\rho(x)^{\beta -j}$, and
$[\nabla^k f|_{U_2}\bigr]_\alpha \leq N_1 2^{k+\alpha-\beta}\rho(x)^{\beta-k-\alpha}$.
Plugging these into (\ref{eq:Sch-bd}) shows that the right hand side of the inequality is bounded by a constant times $\rho(x)^{\gamma}$.
And the inequalities (\ref{eq:Sch-prop-2}) easily follow.
\end{proof}

The proposition is now proved by repeatedly applying the lemma.  Let $\gamma =-2n+2+\delta$.
Let $p$ be the smallest integer with $-p\gamma >k+\alpha$ and define $\lambda_i =\frac{-i\gamma}{k+\alpha}$ for
$i=0,\ldots,n-1$ and $\lambda_p =1$.  We have already shown that (\ref{eq:Sch-prop-1}) holds for $\lambda_0 =0$.
For the inductive step suppose we have (\ref{eq:Sch-prop-1}) for $\lambda=\lambda_i$.  Then we have (\ref{eq:Sch-prop-2}),
and it is easy to see that this implies (\ref{eq:Sch-prop-1}) for $\lambda_{i+1}$.
Thus we get (\ref{eq:Sch-prop-2}) for $\lambda_p =1$, and this implies the proposition.
\end{proof}

In particular, we have
\begin{equation}
 \| \nabla^j (g-g_0 )\|_{g_0} =O(\rho ^{-2n -j+\delta}),\quad\text{for}\quad j\leq k.
\end{equation}
And this completes the proof of Theorem~\ref{thm:main}.

\begin{remark}
One can remove the ``$\delta$'' in Proposition~\ref{prop:decay} by studying the Laplacian on weighted H\"{o}lder spaces on manifolds
with a conical end.  Thus the Ricci-flat metric decays to the Calabi Ansatz metric $g_0$ as in the ALE case~\cite{Joy}.
The details are in~\cite{vC4}.
\end{remark}

\section{Deformations of Gorenstein toric singularities}\label{sec:deform}

In this section $X_\sigma$ is a Gorenstein affine toric variety associated to the strictly convex rational polyhedral cone
$\sigma\subset\R^{n}$.  As in Section~\ref{subsect:toric}, $X_\sigma$ has a K\"{a}hler cone structure and as such $\sigma$ is precisely
the dual moment cone $\mathcal{C}(\mu)^*$ of (\ref{eq:dual-cone}).  Thus $\sigma$ is the span of the $\lambda_i \in\Z^{n},i=1,\ldots,d$, defining
the moment cone in (\ref{eq:moment-cone}).  And Gorenstein means that Proposition~\ref{prop:CY-cond-toric}
is satisfied with $\gamma\in (\Z^{n})^* =\Hom(\Z^{n},\Z)$.  It follows that $\sigma$ is the cone over a lattice polytope $P\subset\R^{n-1}$.

We will assume that $X_\sigma$ is toric with an isolated Gorenstein singularity.  In this case there is a beautiful description by
K. Altmann~\cite{Alt} of the versal deformation space of $X_\sigma$.  The components of the reduced versal space of deformations
of $X_\sigma$ are given by the possible maximal decompositions of $P$ into Minkowski summands.

We give a summary of Altmann's construction.  Let $N=\Z^{n}$ and $M=\Hom (N,\Z)$.  Then $\sigma$ is a cone over an
integral polytope $P\subset N_{\R}$ in the affine plane $\{\ell =-1$, where $\ell\in M$ is that given in
Proposition~\ref{prop:CY-cond-toric}.  We choose an element $n_0 \in N$ with $\ell(n_0)=-1$, so we can identify
$P$ with $P-n_0$ in the hyperplane $L_{\R} \subset N_{\R}$ where $L=\ell^{\perp} \subset N$.

A \emph{Minkowski decomposition} of $P$ in $L_{\R}$ is collection of integral convex subsets $R_0 ,\ldots,R_p$ of $L$ such
that
\[ P=\{r_0 +\cdots+ r_p :r_i \in R_i \}.\]
We write $P=R_0 +\cdots +R_p$.

A maximal decomposition of $P$ into a Minkowski sum $P=R_0 +\cdots +R_p$ with lattice summands $R_k \subset\R^{n-1}$ corresponds
to a reduced irreducible component $\mathscr{M}_0$ of the versal deformation space $\mathscr{M}$ of $X_\sigma$~\cite{Alt}.
The component $\mathscr{M}_0$ is isomorphic to $\C^{p+1}/{\C\cdot (1,\ldots,1)}$, and the restriction
$\pi: Y\rightarrow\C^p$ of the versal family is described as follows.
Define $N'=L\oplus\Z^{p+1}$, and let $e_0, \ldots e_p$ be the usual basis of
$\Z^{p+1}$.  we have the tautological cone in $N'_{\R}$

\begin{equation}
\tilde{\sigma} :=\cone\left(\bigcup_{k=0}^{p} (R_k \times \{e_k\}\right) \subset\R^{n+p},
\end{equation}
where $\cone(S)$ denotes the cone generated by the set $S\subset N'_{\R}$.
We have that $\tilde{\sigma}$ contains $\sigma =\cone(P\times\{1\})\subset\R^n$ via the diagonal embedding
$\R^n \hookrightarrow\R^{n+p}$ ($(s,1)\mapsto (s;1,\ldots,1)$).  The inclusion $\sigma\subseteq\tilde{\sigma}$
induces an embedding of the affine toric variety $X_\sigma \hookrightarrow Y_{\tilde{\sigma}}$.
The projection $\R^{n+p} \rightarrow\R^{p+1}$ induces a map of toric varieties
$Y_{\tilde{\sigma}}\rightarrow\C^{p+1}$, which provides $p+1$ regular functions $t_0,\ldots,t_p$ on $Y_{\tilde{\sigma}}$.
Then the family $\pi: Y\rightarrow\C^p$ is the composition
$Y_{\tilde{\sigma}}\rightarrow\C^{p+1} \rightarrow \C^{p+1}/{\C\cdot (1,\ldots,1)}$.
Thus $\pi$ is given by $(t_1 -t_0, \ldots, t_p -t_0)$.  Given $\ol{c}=(c_1,\ldots,c_p)\in\C^p$ we will denote
$\pi^{-1}(\ol{c})$ by $X_{\ol{c}}$.

\begin{figure}[tbh]
 \centering
 \includegraphics[scale=0.4]{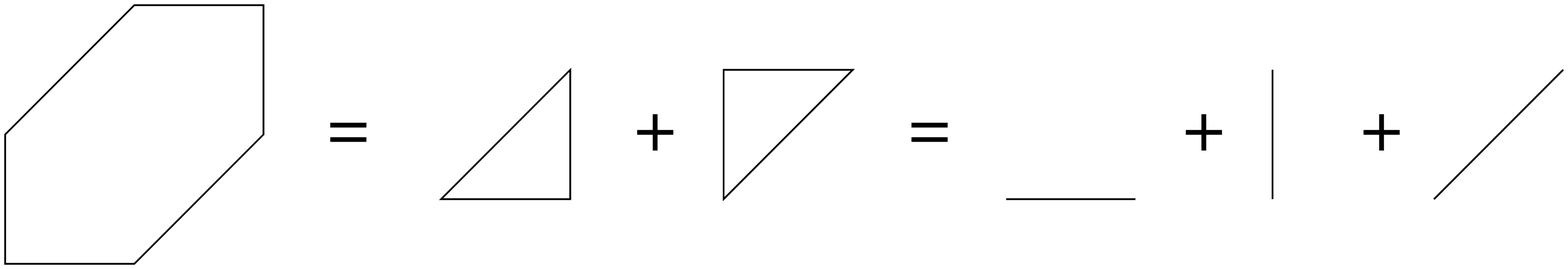}
 \caption{The cone over $\cps^2_{(3)}$}
 \label{fig:Minkow2}
\end{figure}

\begin{figure}[tbh]
 \centering
 \includegraphics[scale=0.4]{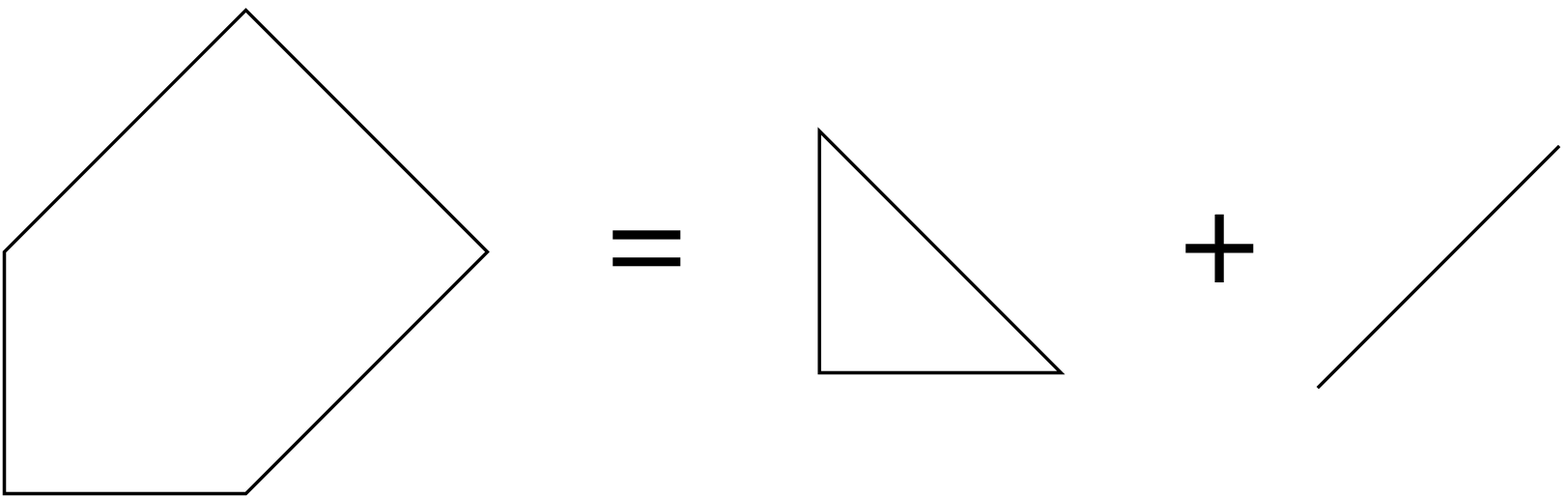}
 \caption{The cone over $\cps^2_{(2)}$}
 \label{fig:Minkow3}
\end{figure}

Figure~\ref{fig:Minkow2} gives the two Minkowski decompositions of the polytope $P$ whose cone $\sigma =\cone(P\times\{1\})$ is the
cone over the degree 6 del Pezzo surface, i.e. $\cps^{(3)}$ the 3 points blow-up.  The polytope giving the cone over
the 2 points blow-up has a single Minkowski decomposition given in Figure~\ref{fig:Minkow3}.

The functions $t_0,\ldots,t_p$ are invariant under the complex $n-1$-torus $T_{\C}(L)\subset T_{\C}(N')$.  Thus
we have a $(\C^*)^{n-1}\cong T_{\C}(L)$ action on $X_{\ol{c}}$.  Note that the cone $\tilde{\sigma}$ is Gorenstein,
since $\ell'\in M'$ with $\ell' =e_0^* +\cdots+ e_p^*$ evaluates to 1 on all the generators.
So there is a nowhere vanishing holomorphic $n+p$-form $\Omega'$ on $Y_{\tilde{\sigma}}$
as in Proposition~\ref{prop:CY-cond-toric}.

Let $\Sigma$ be any fan in $N=\Z^{n}$.  Recall that there a bijective correspondence between $p$-dimensional cones $\sigma\in\Sigma$
and orbits $T_{\sigma} \subset X_{\Sigma}$ of codimension $p$ isomorphic to $T^{n-p}_{\C}$.

\begin{defn}
 We say that a codimension $p$ analytic subvariety $V$ of a toric variety $X_\Sigma$ is $\Sigma$-regular if for
every $\sigma\in\Sigma$ the intersection $V\cap T_{\sigma}$ is a smooth codimension $p$ subvariety.
\end{defn}

\begin{prop}\label{prop:reg}
 A $\Sigma$-regular subvariety of a smooth toric variety $V\subset X_\Sigma$ is smooth.
\end{prop}
\begin{proof}
We denote by $\mathbf{A}_{\sigma}\subset X_\Sigma$ the $n$-dimensional affine toric variety associated to the cone $\sigma\in\Sigma$.
If $\sigma\in\Sigma(r)$ is r-dimensional, let $N_\sigma$ be the r-dimensional sublattice containing $\sigma\cap N$.
Denote by $\mathbf{A}_{\sigma, N(\sigma)}$ the r-dimensional affine toric variety corresponding to $\sigma\subset N(\sigma)$.
Then $\mathbf{A}_\sigma =\mathbf{A}_{\sigma, N(\sigma)} \times (\C^* )^{n-r}$.  The implicit function theorem implies that $V$ is smooth
at every point of $V\cap T_{\sigma} =V\cap \{o\}\times (\C^* )^{n-r} \subset \mathbf{A}_{\sigma, N(\sigma)} \times (\C^* )^{n-r}$.
\end{proof}

\begin{prop}\label{prop:smooth}
 Suppose $Y_{\tilde{\sigma}}$ is smooth besides the one isolated singularity corresponding to the $n+p$-dimensional cone, i.e.
$\tilde{\sigma}$ satisfies (\ref{eq:smooth}).  Then for generic $\ol{c}=(c_1,\ldots,c_p)\in\C^p$, the subvariety
$\pi^{-1}(\ol{c})=:Y_{\ol{c}}\subset Y_{\tilde{\sigma}}$ is smooth.
\end{prop}
\begin{proof}
Let $e^*_k \in\tilde{\sigma}^\vee$ be the extension of the standard dual basis element on $\R^{p+1}$.  For $\tau\in\tilde{\sigma}$,
$t_k$ vanishes on $T_{\tau}$ unless $e^*_k \in\tau^{\perp}$.  Suppose $e^*_{k_1},e^*_{k_2}\notin\tau^{\perp}$.  Then $t_{k_1} -t_{k_2}$
vanishes on $T_{\tau}$.  So $Y_{\ol{c}} \cap T_{\tau} =\emptyset$ if $c_{k_1} -c_{k_2} \neq 0$.
Thus we may assume that only one $e^*_k \notin\tau^{\perp}$.  Then the
$t_1 -t_0 =c_1,\ldots, t_{k-1} -t_{0}=c_{k-1}, -t_0 =c_k, t_{k+1} -t_{0} =c_{k+1},\ldots ,t_p -t_0 =c_p$ define $Y_{\ol{c}} \cap T_{\tau}$
which is a complete intersection for sufficiently general $c_1,\ldots,c_p$.
The result then follows from Proposition~\ref{prop:reg}.
\end{proof}
Notice that for generic $\ol{c}\in\C^p$, as in the proposition, $Y_{\ol{c}} \cap T_\tau =\emptyset$ if $\dim T_\tau <p$.

Supposing that $X_{\ol{c}}\subset Y_{\tilde{\sigma}}$ is smooth, adjunction gives a nowhere vanishing $n$-form
$\Omega$ on $X_{\ol{c}}$.  In order to prove Corollary~\ref{cor:main} we will compactify $Y_{\tilde{\sigma}}$.
Assume that $n_0 \in\inter(P)\cap L$, and define $\tau =-n_0 -e_0 -\cdots -e_p$.
We define a new fan $\Sigma$ in $N'$.  For each face $\mathcal{F}$ of $\tilde{\sigma}$ spanned by
$\beta_1 ,\ldots, \beta_r$ $\Sigma$ contains the cone
$\R_{\geq 0} \tau +\R_{\geq 0} \beta_1 +\cdots +\R_{\geq 0} \beta_r$.  If the fan $\tilde{\sigma}(n+p-1)$,
ignoring the $n+p$-dimensional cone, is simplicial, then $\Sigma$ is simplicial apart for the $n+p$-cone
spanned by $\tilde{\sigma}$.  Thus apart from the singular point corresponding to this cone,
$\ol{Y}_{\tilde{\sigma}}:= Y_\Sigma$ has an orbifold structure.  And $\ol{Y}_{\tilde{\sigma}}$ is a compactification
of $Y_{\tilde{\sigma}}$ by adding a divisor $D_\tau$ at infinity.  Even if $\Sigma$ is not simplicial
$D_\tau$ is $\Q$-Cartier.
Similarly, we define $\ol{X}_\sigma$ to be the variety constructed from $X_\sigma$ as above by adding $-n_0 \in N$.
Then $\ol{X}_\sigma$ is a compactification which is an orbifold apart from the singular point given
by the $n$-cone of $\sigma$.  We have the embedding
$\ol{X}_\sigma \hookrightarrow \ol{Y}_{\tilde{\sigma}}$.  And we define
$\ol{X}_{\ol{c}} \subset\ol{Y}_{\tilde{\sigma}}$ to be the closure of $X_{\ol{c}}\subset\ol{Y}_{\tilde{\sigma}}$.

The regular functions $f_1 =t_1 -t_0 ,\ldots,f_p =t_p -t_0$ on $Y_{\tilde{\sigma}}$ extend to
rational functions on $\ol{Y}_{\tilde{\sigma}}$.  Suppose that Proposition~\ref{prop:smooth} is
satisfied for $\ol{c}\in\C^p$.  Then the Weil divisors $\{f_i =c_i\}\subset Y_{\tilde{\sigma}}, i=1,\ldots,p$
intersect transversally.  Let $D_i$ be the Weil divisor which is the closure of $\{f_i =c_i\}$ in
$\ol{Y}_{\tilde{\sigma}}$.  Since $f_i -c_i$ has a pole along $D_\tau$ and for generic $\ol{c}$
$\{f_i =c_i\}$ contains no other $T_{\C}(N')$-invariant divisors, we have
$D_\tau =D_i -(f_i -c_i)$.  The holomorphic $n+p$-form $\Omega'$ on $\ol{Y}_{\tilde{\sigma}}$ has a
pole of order $p+2$ along $D_{\tau}$.  By applying the adjunction formula p times
we get a holomorphic n-form $\Omega$ on $\ol{X}_{\ol{c}}$ with a pole of order 2 along
$D_\tau \cap\ol{X}_{\ol{c}}$.  It is also not difficult to see that
$D:=D_\tau \cap\ol{X}_{\ol{c}}=D_\tau \cap\ol{X}_\sigma$.

\begin{prop}
Suppose that $X_{\ol{c}}$ is smooth.  Then $\ol{X}_{\ol{c}}$ is an orbifold.  And we have
$\mathbf{K}_{\ol{X}_{\ol{c}}} =[-2D]$.  If $n=3$ then condition (\ref{cond}) holds for $D\subset\ol{X}_{\ol{c}}$.
\end{prop}
\begin{proof}
Proposition~\ref{prop:cond}
is not applicable as $\ol{X}_{\ol{c}}$ and $D$ are orbifolds.  But condition (\ref{cond}) does hold, where the sheaves are
coherent sheaves of orbifold bundles.  By the argument in Proposition~\ref{prop:cond} it is sufficient to prove that
$H^1(D,\Omega^1((k-1)D))=0$ and $H^1(D,\mathcal{O}((1-k)D))=0$ for all $k\geq 2$.  The second holds by the negativity
of $[(1-k)D]$.  Let $\sum_{i=1}^d C_i$, be the anti-canonical divisor of $D$.  Consider the exact sequence
\begin{equation}
0\rightarrow\Omega^1_D \rightarrow\mathcal{O}_D \oplus\mathcal{O}_D \rightarrow\bigoplus_{i=1}^d \mathcal{O}_{C_i}\rightarrow 0,
\end{equation}
where $\Omega^1_D$ is the sheaf of sections of the orbifold bundle of holomorphic 1-forms.  Then tensor with
$\mathbf{E}=[(1-k)D]|_D$, to get
\begin{equation}
0\rightarrow\Omega^1_D (\mathbf{E})\rightarrow\mathcal{O}_D(\mathbf{E}) \oplus\mathcal{O}_D (\mathbf{E})\rightarrow\bigoplus_{i=1}^d \mathcal{O}_{C_i}(\mathbf{E})\rightarrow 0.
\end{equation}
Since $\mathbf{E}$ is negative, Kodaira vanishing and the cohomology sequence gives
\begin{equation}
H^1(D,\Omega^1((k-1)D))=H^1(D,\Omega^1 (\mathbf{E}))=0,
\end{equation}
where the first equality is Serre duality.
\end{proof}

It is proved in~\cite{Alt} that isolated Gorenstein singularities are rigid for $n\geq 4$.
We assumed that there is an $n_0 \in\inter(P)\cap N$.  The only isolated Gorenstein singularity with
$n\leq 3$ for which this is not the case is the quadric cone $X=\{(V,W,Y,Z)\in\C^4 :VW-YZ=0\}$.
And $X$ admits a one dimensional versal space of deformations (cf.~\cite{Tjur}),
$X_\epsilon =\{(V,W,Y,Z)\in\C^4 :VW-YZ=\epsilon \}$.  For $\epsilon\neq 0$, $X_\epsilon$ is smooth and diffeomorphic
to $T^* S^3$.  In fact, with the symplectic structure as a submanifold of $\C^4$, $X_\epsilon$ is symplectomorphic to
$T^* S^3$ with the canonical symplectic structure.  The complex structure is that given as the complexification $SL(2,\C)$ of $S^3$.
In~\cite{Ste} a Ricci-flat K\"{a}hler metric is constructed on $X_\epsilon$ which is asymptotic to an obvious cone metric
on $T^* S^3$ as in (\ref{eq:asymp}) with exponent $-3$ rather than $-2n+\delta$.

For $n=2$ all toric Gorenstein singularities $X$ are cyclic quotients of $\C^2$ by $\Z_{k+1} \subset SL(2,\C)$,
where $\Z_{k+1}$ is generated by $(\alpha, \alpha^k)$ with $\alpha$ a primitive $k+1$-th root of unity.
Thus it is the rational double point $A_k$ which we denote $X_k$.  Recall
$X_k =\{(x,y,z)\in\C^3 : xy+z^{k+1} =0\}.$   The versal deformation space of $X_k$ is easy to construct~\cite{Tjur}.
The versal deformation space of $X_k$ is the subspace
$Y_k\subset\C^3 \times\C^k(t_1,\ldots,t_k)$ defined by the equation
\begin{equation}
F=xy+z^{k+1} +t_1 z^{k-1} +\cdots +t_{k-1}z +t_k =0.
\end{equation}
And the projection $\pi: Y_k \rightarrow \C^k(t_1,\ldots,t_k)$ is the restriction of the obvious projection
$Y_k\subset\C^3 \times\C^k(t_1,\ldots,t_k)\rightarrow\C^k(t_1,\ldots,t_k)$.
We compactify $X_k(t_1,\ldots,t_k)=\pi^{-1}(t_1,\ldots,t_k)$ as follows.
For $k$ even, consider the weighted projective space $\cps^3(k+1,k+1,2,1)$ with homogeneous coordinates
$[x:y:z:s]$ and weights $\mathbf{w}=(w_0,w_1,w_2,w_3)=(k+1,k+1,2,1)$.  Then define
$\ol{X}_k(t_1,\ldots,t_k)\subset\cps^3(k+1,k+1,2,1)$ to be the hyperpersurface
\begin{equation}\label{eq:weig-comp}
f(x,y,z,s)=xy +z^{k+1} +t_1 z^{k-1}s^4 +t_2 z^{k-2} s^6 +\cdots +t_{k-1}zs^{2k} +t_k s^{2k+2}=0.
\end{equation}
Since $f$ is weighted homogeneous of degree $d=2k+2$ and $\sum_i w_i -d=3>0$, we have that $\mathbf{K}^{-1}_{\ol{X}}>0$.
If we denote $D=\{s=0\}\cap \ol{X}$, then $-K_{\ol{X}} =3[D]$.  And $D$ is $\cps^1$ with two antipodal orbifold points of
degree $k+1$.  Thus $D$ admits a K\"{a}hler-Einstein metric, i.e. the $\Z_{k+1}$ quotient of the constant curvature metric.
When $X_k(t_1,\ldots, t_k)$ is smooth, Theorem~\ref{thm:C-Y} is applicable and gives the required metric.
The case with $k$ odd is similar.  Just use the weights $\mathbf{w}=(\frac{k+1}{2},\frac{k+1}{2},1,1)$ and (\ref{eq:weig-comp})
with the necessary changes.  This completes the proof of Corollary~\ref{cor:main}.

Suppose that a toric Gorenstein singularity $X$, with $\dim_{\C} X=3$, admits a smoothing which we
denote $\hat{X}$.  It was proved in~\cite{vC2,vC3} that $X$ also has a resolution $\check{X}$ which
admits a Ricci-flat K\"{a}hler metric asymptotic to a cone metric.
Thus we have two smooth Calabi-Yau varieties $\check{X}$ and $\hat{X}$ which admit Ricci-flat K\"{a}hler metrics.
We can transition $\check{X}\leadsto X$ by shrinking the exceptional divisor to a point.  And we can deform
$X\leadsto\hat{X}$.  This is a \emph{geometric transition} analogous to that given by the small resolution
and smoothing of the conifold $\{(U,V,X,Y)\in\C^4 : UV-XY=0\}$ which has been investigated in~\cite{CanOs}
and elsewhere.  See~\cite{Ros} for more on geometric transitions
in the study of Calabi-Yau manifolds and their relevance to physics.  Since
the resolved space $\check{X}$ is toric, and also the smoothing $\hat{Y}$, when it exists,
admits a $(\C^*)^{n-1}$ action (cf.~\cite{Alt}),
it follows from the work of M. Gross~\cite{Gro} that both $\check{X}$ and $\hat{X}$ admit special Lagrangian
fibrations.  Thus this phenomenon should be of interest in the Strominger-Yau-Zaslow conjecture and mirror symmetry.

Though a toric Gorenstein singularity $X$, with $\dim_{\C} X =3$, always admits a resolution $\check{X}$ with
a Ricci-flat K\"{a}hler metric, a smoothing $\hat{X}$ may not exist.  For example, if as in
Example~\ref{xpl:toric} one takes $M=\cps^2_{(1)}$, then the total space
$\check{Y}=\mathbf{K}_M$ is a resolution of the cone $X=\mathbf{K}^\times _M \cup\{o\}$.  And $\check{Y}$
admits a Ricci-flat K\"{a}hler metric.  But it is proved in~\cite{Alt} that the cone $X$ is rigid.

\section{Examples}\label{sec:examp}

\begin{xpl}\label{xpl:toric}
Let $M$ be a toric Fano manifold.  Let $X=\mathbb{P}(\mathbf{K}_M \oplus\underline{\C})$.
If $D\subset X$ is the $\infty$-section of $\mathbf{K}_M$, then $2[D]=-K_X$.  Note that $X$ is not a Fano manifold,
but $D$ is a good divisor as in Definition~\ref{defn:good-div}.   The arguments in Proposition~\ref{prop:cond}
show that condition (\ref{cond}) holds.  But this is immaterial as it is clear that the normal bundle $N_D$ is
biholomorphic to a neighborhood of $D$ in $X$ in a obvious way.
Theorem~\ref{thm:FOW} implies the $U(1)$-subbundle $S\subset\mathbf{K}_M$ admits a Sasaki-Einstein structure.
And Theorem~\ref{thm:main} implies that $\mathbf{K}_M$ admits a complete Ricci-flat
K\"{a}hler metric which converges to the Calabi ansatz at infinity.

Of course, if $M$ admits a K\"{a}hler-Einstein metric then the Calabi ansatz~\cite{Cal} constructs a complete
Ricci-flat K\"{a}hler metric on $\mathbf{K}_M$ as in section~\ref{subsect:calabi} which is explicit up to the
K\"{a}hler-Einstein metric on $M$.
The problem of the existence of a K\"{a}hler-Einstein metric on a toric Fano manifold was solved in~\cite{WaZh}, where
it was proved that the only obstruction is the Futaki invariant.
We saw that the Calabi ansatz always constructs a Ricci-flat metric in a neighborhood of infinity on $\mathbf{K}_M$.
But the author does not believe that this metric extends smoothly across the zero section in the case when $M$ does not
admit a K\"{a}hler-Einstein metric.

In particular, suppose $M=\cps^2_{(2)}$ the two-points blow-up.  Then the total space $Y=\mathbf{K}_M$ admits a complete Ricci-flat
K\"{a}hler metric which converges to the Calabi ansatz at infinity.  Note that $Y$ is also a toric variety.
Thus as an analytic variety $Y$ is described by a fan $\Delta$ in $\R^3$ which in this situation is a cone over a triagulated polytope in
$P\subset\R^2 =\{(x_1 ,x_2 ,x_3)\in\R^3 :x_3=1\}$.  The polytope $P$ and its triangulation are shown in Figure~\ref{fig:can-bund}.
The cone over $P$ is just the dual cone $\mathcal{C}^\vee$ to the moment polytope of $\mathbf{K}_M^\times =C(S)$ where
$S$ is the Sasaki-Einstein manifold of Example~\ref{xpl:FOW}.  Note that collapsing the zero section of $\mathbf{K}_M$ gives a morphism
$\pi:Y\rightarrow C(S)\cup\{o\}$.

\begin{figure}[tbh]
 \centering
 \includegraphics[scale=0.4]{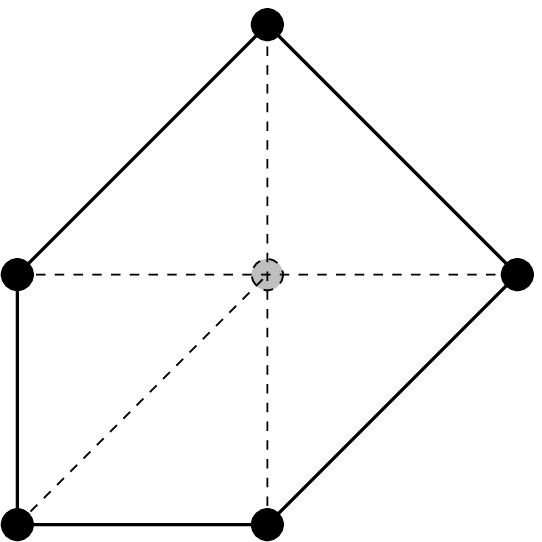}
 \caption{Canonical bundle of $\cps^2_{(2)}$}
 \label{fig:can-bund}
\end{figure}

\end{xpl}

\begin{xpl}\label{xpl:defo}
Consider the Fano 3-fold $X=\mathbb{P}(\mathcal{O}_{\cps^2}(1)\oplus\underline{\C}) =\cps^3_{(1)}$, the one-point blow-up of
$\cps^3$.  This is $V_7$ in the classification of Fano 3-folds of V. A. Iskovskikh~\cite{PS}.
Then $\ind(X)=2$, and there exists a smooth subvariety $D\subset X$ with $2D=-K_X$.
Since $-K_X^3 =56$,  one easily sees that $K_D^2 =7$.  But since $D$ is a del Pezzo surface, we must have
$D=\cps_{(2)}^2$, the two-points blow-up.  Then Proposition~\ref{prop:cond} implies that condition (\ref{cond})
is satisfies and by Theorem~\ref{thm:main} $Y=X\setminus D$ admits a complete Ricci-flat K\"{a}hler metric.

More explicitly, if $\pi:\cps^3_{(1)} \rightarrow\cps^3$ is the blow-down collapsing the exceptional divisor $E$ to $p\in\cps^3$,
then $-K_X =4\pi^* H -2E$.  We have $D= 2\pi^* H -E$.  Then $D$ can be represented by the strict transform of a smooth quadric surface
through $p$.  So $D=\cps^1 \times\cps^1_{(1)} =\cps_{(2)}^2$.

Notice that the end of $Y$ is diffeomorphic to a cone over the $U(1)$-bundle  $M\subset\mathbf{K}_D$, and
$M$ is diffeomorphic to $(S^2\times S^3)\# (S^2\times S^3)$ by the S. Smale's classification
of smooth simply connected spin 5-manifolds.
This is the only example of a smooth Fano 3-fold $X$ with smooth divisor $D\subset X$, $\alpha D=-K_X$, for which
$D$ does not admit a K\"{a}hler-Einstein metric.
It is not difficult to show that $Y$ is simply connected, and the Betti numbers are $b_0(Y)=b_2(Y)=b_3(Y)=1$ with the rest zero.

The topology of $Y$ can be described in more detail.
The compliment of a smooth quadric $Q\subset\cps^3$ is diffeomorphic to $T^* S^3/\Z_2$, the quotient of the
cotangent bundle of $S^3$ where $\Z_2$ acts by the antipodal map on $S^3$.  This is easily seen to be
$T^*\rps^3$.  Let $W$ be the unit disk bundle in $T^* \rps^3$.  This is the trivial bundle and
$\partial W =\rps^3 \times S^2$.  Take a smooth $S^1$ in $\partial W$ generating $\pi_1(\partial W)$.
It has a neighborhood $N\cong S^1 \times B^4$ in $\partial W$.  Then
$\partial (\partial W\setminus N)= S^1 \times S^3$.  Now glue $B^2 \times B^4$ to $W\setminus N$ along
$S^1 \times S^3\subset B^2 \times B^4$.  We have $Y \cong (W\setminus N)\cup_{S^1 \times S^3} B^2 \times B^4$.
It follows that $Y$ is homotopy equivalent to $\rps^3 \cup_{f}B^2$, where $f: S^1 \rightarrow\rps^3$
generates $\pi_1(\rps^3)$.  If $X$ denotes the cone over $\cps^2_{(2)}$, i.e. the toric variety given by a fan
which is the cone over the polytope $P$ in Figure~\ref{fig:can-bund}, then topologically we have replaced the vertex
with $\rps^3 \cup_{f}B^2$.

This example is a smoothing $\hat{X}$ of the toric Gorenstein variety $X$ in Example~\ref{xpl:FOW} as discussed in
Section~\ref{sec:deform}.  This can be described geometrically.  The divisor $D= 2\pi^* H -E$ on $\cps^3_{(1)}$ is very ample.
This follows from the
result that an ample divisor on a non-singular toric variety is very ample~\cite[Theorem 2.22]{Oda}.  And
$H^0(\cps^3_{(1)},\mathcal{O}(D))$ is spanned by 9 sections.  Thus we have an embedding
$\iota_{[D]}:\cps^3_{(1)}\rightarrow \cps^8$.  Let $V\subset\C^9$ be the cone over this variety.
If $H\subset\C^9$ is a generic hyperplane through $o\in\C^9$, then $X=H\cap V$.
Generically deform $H$ away from $o\in\C^9$, then $H\cap V =\hat{X}$.  This deformation of $X$ is
given by the Minkowski decomposition of $P$ in Figure~\ref{fig:Minkow3}.
\end{xpl}

Examples~\ref{xpl:toric} and~\ref{xpl:defo} provide an example of a geometric transition.
Let $X=C(S)\cup\{o\}$ be the Ricci-flat toric K\"{a}hler cone which  is given by the fan
which is the cone over the polytope $P$ in Figure~\ref{fig:can-bund}.  Let $\check{X}$ denote the toric
resolution of Example~\ref{xpl:toric}, and $\hat{X}$ the smoothing of $X$ of Example~\ref{xpl:defo}.
We can transition $\check{X}\leadsto X$ by shrinking the exceptional $\cps_{(2)}^2$ to a point.
And then we can deform $X\leadsto\hat{X}$.

\begin{xpl}\label{xpl:sym-def}
We consider a series of examples of toric singularities with a symmetry property.  The symmetry
ensures that the polytope $P$ is a Minkowski sum with a large number of components.  These examples
arose in the author's investigation of submanifolds of toric 3-Sasakian manifolds~\cite{vC,vC1}.

Let $P\subset\R^2$ be an integral polytope invariant under the antipodal map
$(x_1 ,x_2)\mapsto(-x_1 ,-x_2)$.  Then considering $P\subset \R^2 \times \{(0,0,1)\}\subset\R^3$,
$\sigma =\cone(P)$.  We require that $X_{\sigma}$ has only an isolated singularity.
The variety $X_\sigma \setminus\{o\}$, minus the singular point, has $\pi_1 (X_\sigma \setminus\{o\})=e,\Z_2$.
If $\pi_1 (X_\sigma \setminus\{o\})=\Z_2$,
then take the sublattice $N'\subset\Z^3$ generated by the $\{(x_1 ,x_2, 1)\}$ where $(x_1 ,x_2)$
is a vertex of $P$.  Denote by $\sigma'$ the cone $\sigma$ as a cone in $N'$.  Then
$X_{\sigma'} \setminus\{o\}$ is simply connected and $\sigma'$ has the essential property used below.
It was shown in~\cite{vC,vC1}, by considering submanifolds of toric 3-Sasakian 7-manifolds, how to produce
infinite families of $\sigma'$ with this property.

Let $d_0, \ldots, d_k \in\Z^2$ be the edge vectors of $P$ taken counterclockwise around $P$.
Thus if $\lambda_0,\cdots,\lambda_k$ span $\sigma'$, then $d_j =\lambda_j -\lambda_{j-1}$.
These are primitive, since $X_{\sigma'}$ has only an isolated singularity.
Then $k+1 =2(p+1)$ and $d_{j+p+1} =-d_j$, where indices are taken $\mod k+1$.  It follows that
the elements $d_0,\ldots, d_p$ give a Minkowski decomposition of $P$ (cf.~\cite{Alt}, \S 2).
More precisely, if $R_j \subset\R^2$ is the segment spanned by $(0,0)$ and $d_j$, then
\begin{equation}
P=R_0 +\cdots +R_p.
\end{equation}

\begin{prop}
With $\sigma$ as above, the tautological cone $\tilde{\sigma}$ satisfies the smoothness condition
(\ref{eq:smooth}), i.e. $Y_{\tilde{\sigma}}$ has only an isolated singularity.
\end{prop}
\begin{proof}
Let $\alpha\in\tilde{\sigma}^{\vee}$.  Then $\alpha =\beta +\sum_{j=0}^p c_j e_j^*$ where
$\beta\in (\R^2)^*$ and $c_j \geq 0, 0\leq j\leq p$.  The cone $\tilde{\sigma}$ is generated
by $\tau_j =( 0, e_j)$ and $\tau'_j =(d_j, e_j)$ for $0\leq j\leq p$.  Thus
$\beta(d_j)+c_j \geq 0$ for $0\leq j\leq p$.

A facet of $\tilde{\sigma}$ is given by choice of $\alpha$ vanishing on $p+2$ generators.
The only possibilities are $\{\tau'_0,\ldots,\tau'_i,\tau_i,\ldots,\tau_p \}$, where
$\beta(d_i) =0$, $c_j =0, i \leq j\leq p$, and $\beta(d_j) =-c_j$ for $0\leq j\leq i$.
And it is easy to see that the face spanned by this set satisfies (\ref{eq:smooth}).
\end{proof}

By Proposition~\ref{prop:smooth} for generic $\ol{c}\in\C^p$ the affine variety
$X_{\ol{c}}\subset Y_{\tilde{\sigma}}$ is smooth.  And by Corollary~\ref{cor:main} there
is a complete Ricci-flat K\"{a}hler metric asymptotic to a Calabi ansatz metric.

This provides infinitely many examples of geometric transitions.  As shown in~\cite{vC2,vC3}
the toric cone $X_\sigma$ admits a resolution $\check{X}$ with Ricci-flat K\"{a}hler metric
asymptotic to the K\"{a}hler cone metric on $X_\sigma$ of Theorem~\ref{thm:FOW}.  And $X$ can
also be smoothed to an affine variety $\hat{X}$ with a Ricci-flat metric with similar asymptotic
properties.
\end{xpl}

\bibliographystyle{plain}

\end{document}